\documentclass[final]{dmtcs-episciences}
\usepackage[ansinew]{inputenc}          
\usepackage{amssymb,amsthm,amsmath}     
\usepackage{psfrag}                     
\usepackage[all]{xy}
\usepackage{algorithm,algorithmic}
\usepackage{a4wide}
\usepackage{tikz}
\usepackage{enumerate}
\usepackage{xcolor}
\usepackage{soul}
\usetikzlibrary{matrix}
\usetikzlibrary{decorations.markings}
\usetikzlibrary{positioning}
\tikzset{main node/.style={circle,fill=blue!20,draw,minimum size=0.2cm,inner sep=0pt},
            }
\usepackage[labelformat=simple]{subcaption}

\theoremstyle{plain}
\newtheorem{theorem}{Theorem}

\newtheorem{lemma}[theorem]{Lemma}
\newtheorem{corollary}[theorem]{Corollary}
\newtheorem{proposition}[theorem]{Proposition}
\theoremstyle{definition}
\newtheorem{definition}[theorem]{Definition}
\newtheorem{example}[theorem]{Example}
\newtheorem{remark}[theorem]{Remark}

\newcommand{\Z}{\mathbb{Z}}
\newcommand{\F}{\mathbb{F}}

\newcommand{\LD}{\gamma^{LD}}

\newcommand{\SLD}{\gamma^{SLD}}
\newcommand{\DLD}{\gamma^{DLD}}

\author{Ville Junnila\affiliationmark{1}
  \and Tero Laihonen\affiliationmark{1}
  \and Tuomo Lehtil{\"a}\affiliationmark{1}\thanks{Research supported by the University of Turku Graduate School (UTUGS) and the Vilho, Yrj{\"o} and Kalle V{\"a}is{\"a}l{\"a} Foundation.}
  \and Mar\'{i}a Luz Puertas\affiliationmark{2}\thanks{Author partially supported by the grant MTM2015-63791-R (MINECO-FEDER).}}
\title[Self- and Solid-location-domination]{On Stronger Types of Locating-Dominating Codes\thanks{The paper has been presented in part in the 10th International Colloquium on Graph Theory and Combinatorics (2018, Lyon, France).}}
\affiliation{
  Department of Mathematics and Statistics, University of Turku, Finland\\
  Department of Mathematics, Universidad de Almeria, Spain}
\received{2018-8-22}
\revised{2019-3-12}
\accepted{2019-3-25}

\begin{document}
\publicationdetails{21}{2019}{1}{1}{4771}
\maketitle



\begin{abstract}
Locating-dominating codes  in a graph find their application in sensor networks and have been studied extensively over the years. A locating-dominating code can locate one object in a sensor network, but if there is more than one object, it may lead to false conclusions. In this paper, we consider stronger types of locating-dominating codes which can locate one object and detect if there are multiple objects. We study the properties of these codes and  provide bounds on the
smallest possible size of these codes, for example, with the aid of the Dilworth number and Sperner families. Moreover, these codes are studied in trees and Cartesian products of graphs. We also give the complete realization theorems for the coexistence of the
smallest possible size of these codes and the optimal locating-dominating codes in a graph.
\end{abstract}


\section{Introduction}

Sensor networks are systems designed for environmental monitoring.
Various location detection systems such as fire alarm and
surveillance systems can be viewed as examples of sensor networks.
For location detection, a sensor can be placed in any location of
the network. The sensor monitors its neighbourhood (including the
location of the sensor itself) and reports possible objects or irregularities
such as a fire or an intruder in the neighbouring locations. In the
model considered in the paper, it is assumed that a sensor can
distinguish whether the irregularity is in the location of the
sensor or in the neighbouring locations (as in
~\cite{RS:LDnumber,S:DomLocAcyclic,S:DomandRef}). Based on the
reports of the sensors, a central controller attempts to determine
the location of a possible irregularity in the network. Usually, the
aim is to minimize the number of sensors in the network. More
explanation regarding location detection in sensor networks can be
found in~\cite{Trachtenberg,LT:disj,Ray}. An online bibliography on the topic can be found at \cite{lowww}.

A sensor network can be modelled as a simple and undirected graph $G =(V(G),E(G)) = (V,E)$ as follows: the set of vertices $V$ of the graph represents the locations of the network and the edge set $E$ of the graph represents the connections between the locations. In other words, a sensor can be placed in each vertex of the graph and the sensor placed in the vertex $u$ monitors $u$ itself and the vertices neighbouring $u$. Moreover, besides being simple and undirected we also assume that the graphs in this paper are finite. In what follows, we present some basic terminology and notation regarding graphs. The \emph{open neighbourhood} of $u \in V$ consists of the vertices adjacent to $u$ and it is denoted by $N(u)$. The \emph{closed neighbourhood} of $u$ is defined as $N[u] = \{ u \} \cup N(u)$. The \emph{degree} of a vertex $u$ is the number of vertices in the open neighbourhood $N(u)$ and the \emph{maximum degree} $\Delta(G) = \Delta$ of the graph $G$ is the maximum degree among all the vertices of $G$. The \textit{distance} between two vertices $d(u,v)$ is the number of edges in any shortest path connecting them. A non-empty subset $C$ of $V$ is called a \emph{code} and the elements of the code are called \emph{codewords}. In this paper, the code $C$ (usually) represents the set of locations where the sensors have been placed on. For the set of sensors monitoring a vertex $u \in V$, we use the following notation:
\[
I(u) = N[u] \cap C \text{.}
\]
In order to emphasize the graph $G$ and/or the code $C$, we sometimes write $I(u) = I(C;u) = I(G,C;u)$. We call $I(u)$ the $I$\emph{-set} or the \emph{identifying set} of $u$. 

As stated above, a sensor $u \in V$ reports that an irregularity has been detected if there is (at least) one in the closed neighbourhood $N[u]$. In the model of the paper, we further assume that a sensor $u \in V$ reports $2$ if there is an irregularity in $u$, it reports $1$ if there is one in $N(u)$ (and none in $u$ itself), and otherwise it reports $0$. In other words, a sensor can distinguish whether an irregularity is in the location of the sensor or in the neighbouring locations. We say that a set (or a code) $C$ is \emph{dominating} in $G$ if $I(C;u)$ is non-empty for all $u \in V \setminus C$. In other words, an irregularity in the network can be detected (albeit not located). Furthermore, the smallest cardinality of a dominating set in $G$ is called the \emph{domination number} and it is denoted by $\gamma(G)$. Notice then that if the sensors in the code $C$ are located in such places that $I(C;u)$ is non-empty and unique for all $u \in V \setminus C$, then an irregularity in the network can be located by comparing $I(C;u)$ to $I$-sets of other non-codewords. This leads to the following definition of \emph{locating-dominating codes (or sets)}, which were first introduced by Slater in~\cite{RS:LDnumber,S:DomLocAcyclic,S:DomandRef}.
\begin{definition}
A code $C \subseteq V$ is \emph{locating-dominating} in $G$ if for all distinct $u, v \in V \setminus C$ we have $I(C;u) \neq \emptyset$ and
\[
I(C;u) \neq I(C;v) \text{.}
\]

A locating-dominating code $C$ in a finite graph $G$ with the smallest cardinality is called \emph{optimal} and the number of codewords in an optimal locating-dominating code is denoted by $\LD(G)$. The value $\LD(G)$ is also called the \emph{location-domination number}.
\end{definition}

The previous definition of locating-dominating codes is illustrated in the following example.
\begin{example} \label{ExampleLD}
Let $G$ be the graph illustrated in Figure~\ref{FigureLD}. Consider the code $C = \{b, d, f \}$ in $G$ (see Figure~\ref{FigureLD}). Since the $I$-sets $I(C;a) = \{b, d\}$, $I(C;c) = \{b, f\}$ and $I(C;e) = \{b, d, f\}$ are all non-empty and different, the code $C$ is locating-dominating in $G$. Moreover, there do not exist smaller locating-dominating codes in $G$ as using at most two codewords we can form at most three different non-empty $I$-sets. Therefore, we have $\LD(G) = 3$.
\end{example}

\begin{figure}[ht] 
\centering
 \includegraphics[height=60pt]{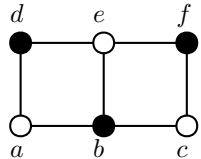}
 \caption{Optimal locating-dominating code in a graph $G$}
 \label{FigureLD}
\end{figure}

The original concept of locating-dominating codes has some issues in certain types of applications. Firstly, locating-dominating codes might output misleading results if there exist more than one irregularity in the graph. For instance, if in the previous example there exist irregularities in $a$ and $c$, then the sensors located at $b$, $d$ and $f$ are reporting $1$. Now the system deduces that the irregularity is in $e$. Hence, a completely false output is given and we do not even notice that something is wrong. Secondly, in order to determine the location of the irregularity, we have to compare the obtained $I$-set to other such sets. In order to overcome these issues, so called self-locating-dominating and solid-locating-dominating codes have been introduced in~\cite{JLLrntcld} motivated by $(1,\leq 1)^+$-identifying or self-identifying codes introduced in \cite{1+ID,SelfIDBor,SelfID}. For more detailed discussion on the motivation of self-locating-dominating and solid-locating-dominating codes, the interested reader is referred to~\cite{JLLrntcld}. 
The formal definitions of these codes are given in the following.
\begin{definition}\label{defSLD}
A code $C \subseteq V$ is \emph{self-locating-dominating} in $G$ if, for all $u \in V \setminus C$, we have $I(C;u)\neq\emptyset$ and
\[
\bigcap_{c \in I(C;u)} N[c] = \{u\} \text{.}
\]
A self-locating-dominating code $C$ in a finite graph $G$ with the smallest cardinality is called \emph{optimal} and the number of codewords in an optimal self-locating-dominating code is denoted by $\SLD(G)$. The value $\SLD(G)$ is also called the \emph{self-location-domination number}.
\end{definition}

\begin{definition}\label{defDLD}
A code $C \subseteq V$ is \emph{solid-locating-dominating} in $G$ if {$I(C;u) \neq \emptyset$ for every $u\in V \setminus C$} and, for all distinct $u, v \in V \setminus C$, we have

\[
I(C;u) \setminus I(C;v) \neq \emptyset  \text{.}
\]
Note that this condition is equivalent with $I(C;u)\not\subseteq I(C;v)$. A solid-locating-dominating code $C$ in a finite graph $G$ with the smallest cardinality is called \emph{optimal} and the number of codewords in an optimal solid-locating-dominating code is denoted by $\DLD(G)$. The value $\DLD(G)$ is also called the \emph{solid-location-domination number}.
\end{definition}

By the previous definitions, it is immediate that any self-locating-dominating and solid-locating-domi\-nating code is also locating-dominating (see also Corollary~\ref{CorollarySLdtoDLD}) and that every graph contains a self-locating-dominating and solid-locating-dominating code. Indeed, $C=V$ is always a self-locating-dominating and solid-locating-dominating code. The definitions are illustrated in the following example. In particular, we show that the given definitions are indeed different.
\begin{example} \label{ExampleSLDDLD}
Let $G$ be a graph illustrated in Figure~\ref{FigureLDSLDDLD}. Let $C$ be a self-locating-dominating code in $G$. Observe first that if $a \notin C$, then $I(C;a) \subseteq \{b, d\}$ and we have
\[
\{a, e\} \subseteq \bigcap_{x \in I(C;a)} N[x] \text{.}
\]
This implies a contradiction and, therefore, the vertex $a$ belongs to $C$. An analogous argument also holds for the vertices $c$, $d$ and $f$. Hence, we have $\{a,c,d,f\} \subseteq C$. Moreover, the code $C_1 = \{a,c,d,f\}$, which is illustrated in Figure~\ref{FigureSLD}, is self-locating-dominating in $G$ since for the non-codewords $b$ and $e$ we have $I(C_1;b) = \{a,c\}$ and $N[a] \cap N[c] = \{b\}$, and $I(C_1;e) = \{d,f\}$ and $N[d] \cap N[f] = \{e\}$. Hence, $C_1$ is an optimal self-locating-dominating code in $G$ and we have $\SLD(G)=4$.

Let us then consider the code $C_2 = \{a,b,c\}$, which is illustrated in Figure~\ref{FigureDLD}. Now we have $I(C_2;d) = \{a \}$, $I(C_2;e) = \{b \}$ and $I(C_2;f) = \{c \}$. Therefore, it is easy to see that $C_2$ is a solid-locating-dominating code in $G$. Moreover, there are no solid-locating-dominating codes in $G$ with smaller number of codewords since even a regular locating-dominating code has always at least $3$ codewords by Example~\ref{ExampleLD}. Thus, $C_2$ is an optimal solid-locating-dominating code in $G$ and we have $\DLD(G)=3$.
\end{example}



\begin{figure}%
    \centering
    \begin{subfigure}{4cm}
    \includegraphics[width=\linewidth]{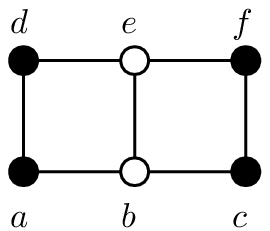}
    \caption{Self-locating-dominating \newline \hspace*{0.42cm} code}\label{FigureSLD}
    \end{subfigure}
    \hspace{1cm}
    \begin{subfigure}{4.18cm}
    \includegraphics[width=4cm]{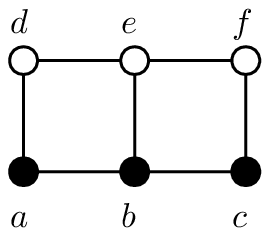}
    \caption{Solid-locating-dominating \newline \hspace*{0.42cm} code}\label{FigureDLD}
    \end{subfigure}
    \caption{Optimal self-locating-dominating and solid-locating-dominating codes in a graph $G$}%
    \label{FigureLDSLDDLD}%
\end{figure}

In the previous example, we showed that the definitions of self-locating-dominating and solid-locating-dominating codes are different. Furthermore, by comparing Examples~\ref{ExampleLD} and \ref{ExampleSLDDLD}, we notice that the new codes are also different from the original locating-dominating codes. In the following theorem, we present new characterizations for self-locating-dominating and solid-locating-dominating codes. Comparing these characterizations to the original definitions of the codes, the differences of the codes become apparent. {We omit the proof of the following theorem, because it is proved in}~\cite{JLLrntcld} {for connected graphs and it is easily modified for non-connected ones.}

\newpage
\begin{theorem} \label{ThmCharacterizationSLDDLD}
Let $G$ be a graph on at least two vertices.
\begin{itemize}
\item[(i)] A code $C \subseteq V$ is self-locating-dominating if and only if, for all distinct $u \in V \setminus C$ and $v \in V$, we have
   $    I(C;u) \setminus I(C;v) \neq \emptyset \text{.}$
\item[(ii)] A code $C \subseteq V$ is solid-locating-dominating if and only if, for all $u \in V \setminus C$, we have {$I(C;u)\neq \emptyset$} and
    \[
    \left( \bigcap_{c \in I(C;u)} N[c] \right) \setminus C = \{u\} \text{.}
    \]
\end{itemize}
\end{theorem}

By comparing Definition \ref{defDLD} and Theorem \ref{ThmCharacterizationSLDDLD} (i) we notice that the only difference is that in $I(C;u)\setminus I(C;v)$ vertex $v$ can be a codeword when we consider self-location-domination. Similarly, when we compare Definition \ref{defSLD} and Theorem \ref{ThmCharacterizationSLDDLD} (ii) we notice that the only difference is that in the case of solid-location-domination we omit codewords from the intersection.

The previous theorem together with the definition of solid-locating-dominating codes and the previous observation immediately gives the following corollary.
\begin{corollary} \label{CorollarySLdtoDLD}
The following facts hold for all graphs $G$.
\begin{itemize}
\item If $C$ is a self-locating-dominating code in $G$, then $C$ is also solid-locating-dominating in $G$.
\item If $C$ is a solid-locating-dominating code in $G$, then $C$ is also locating-dominating in $G$.
\end{itemize}
Thus, we have $\LD(G)\leq \DLD(G) \leq \SLD(G)$.
\end{corollary}

As stated earlier, self-locating-dominating and solid-locating-dominating codes have benefits over regular locating-dominating codes; they detect more than one irregularity and locate one irregularity without comparison to other $I$-sets --- for more details, see \cite{JLLrntcld}.

Previously, when self-locating-dominating and solid-locating-dominating codes have been studied, in \cite{JLLrntcld}, the optimal values for $\SLD(K_n\square K_m)$ and $\gamma^{DLD}(K_n\square K_m)$ have been found. Also the a general lower bound for $\SLD(\F_2^n)$ has been given and an infinite family of constructions attaining this bound is presented for suitable values of $n$. Moreover, a general lower bound for $\gamma^{DLD}(\F_2^n)$ is given and this bound is shown to be asymptotically tight as $n$ grows.

In what follows, the structure of the paper is briefly discussed. In Section~\ref{SectionBasics}, we first show some general bounds and properties for self- and solid-locating-dominating codes; in particular, we utilize the Dilworth number and Sperner families. Then, in Section~\ref{SectionTrees}, we consider the codes in trees and determine self-location-domination and solid-location-domination numbers with the help of other graph parameters. In Section~\ref{SectionCartesian}, we consider Cartesian products and give some general bounds for them which are shown to be achieved in the case of ladders and some rook's graphs.
Finally, in Section~\ref{SectionRealization}, we study the existence of graphs when we are given the location-domination number and the self-location-domination or the solid-location-domination number associated with them.

\section{Basics} \label{SectionBasics}

In this section, we present some basic results regarding self-locating-dominating and solid-locating-dom\-inating codes. In particular, we give various lower and upper bounds for such codes. We first begin by giving results which do not take advantage of any properties or parameters of the graph such as the maximum degree or the independence number.  Then, in Section \ref{secSper}, we use the Sperner's Theorem to gain new bounds. Later, in Section \ref{secDil}, we apply the Dilworth number. Finally, in Section \ref{secIndandComp}, we use independence number and consider complements of graphs.

In the following theorem, we begin by giving a simple upper bound for solid-locating-dominating codes in graphs. {It is clear that the discrete graph $D_n$, with $n$ vertices and no edges, satisfies $\DLD(D_n)=n$, because $V(D_n)$ is its unique dominating set. We now focus on graphs with at least one edge.}

\begin{theorem} \label{ThmDLDn-1}
If $G =(V,E)$ is a graph with order $n$ and size $m\geq 1$,
then the code $V \setminus \{u\}$ is solid-locating-dominating in $G$ for any non-isolated vertex $u \in V$. Thus, we have $$\DLD(G) \leq n-1.$$
\end{theorem}
\begin{proof}
{Let $u\in V$ be a non-isolated vertex of $G$, i.e., $N(u)\cap (V\setminus \{ u\})\neq \emptyset$. By the definition, it is immediate that $V \setminus \{u\}$ is a solid-locating-dominating code of $G$. This further implies that $\DLD(G) \leq |V|-1$.}
\end{proof}

The result of the previous theorem can also be interpreted as follows: {in the particular case of graphs with no isolated vertices,} none of the vertices of a graph is forced to be in all the solid-locating-dominating codes of the graph and hence, the same is also true for locating-dominating codes by Corollary \ref{CorollarySLdtoDLD}. However, this is not the case with self-locating-dominating codes. Hence, for future considerations, we define the concept of forced codewords as follows: a vertex $u$ of $G$ is said to be a \emph{forced codeword} regarding self-location-domination if $u$ belongs to all self-locating-dominating codes in $G$. In the following theorem, we give a simple characterization for forced codewords and show that such vertices indeed exist.
\begin{theorem} \label{ThmForcedCodewordsChar}
Let $G = (V,E)$ be a graph. If $|V| = 1$, then the single vertex of the graph $G$ is a forced codeword. Assuming $|V| \geq 2$, a vertex $u \in V$ is a forced codeword regarding self-location-domination if and only if for some vertex $v \in V$ other than $u$ we have $N(u) \subseteq N[v]$.
\end{theorem}
\begin{proof}
{Let $C$ be a self-locating-dominating code in $G$ and $u$ be a
vertex of $V$. If $|V|=1$, then due to the domination the single vertex is
a forced codeword. Assume now that $|V| \geq 2$. Suppose further that there exists another vertex $v \in V$ such that $u \neq v$ and $N(u) \subseteq N[v]$. If
$N(u)=\emptyset$, then again the domination yields that $u$ is a
forced codeword. Suppose that $N(u)\neq \emptyset.$  This implies
that if $u \notin C$, then}
\[
\{u, v\} \subseteq \bigcap_{c \in N(u)} N[c]\subseteq \bigcap_{c\in I(C;u)}N[c]. 
\]
{Therefore, as the previous intersection does not consist of a
single vertex, the vertex $u$ belongs to $C$ and is a forced codeword}.

Suppose then to the contrary that for any vertex $v \in V$ other than $u$ we have $N(u) \nsubseteq N[v]$, i.e., $N(u) \setminus N[v] \neq \emptyset$. Now choosing $C = V \setminus \{u\} (\neq \emptyset)$, we have $I(C;u) = N(u)$ and
\[
\bigcap_{c \in I(C;u)} N[c] = \{u\} \text{.}
\]
Therefore, by the definition, $C$ is a self-locating-dominating code in $G$. Thus, $u$ is not a forced codeword and we have a contradiction with the supposition. This concludes the proof of the theorem.
\end{proof}

By the previous theorem, we immediately observe that there exist graphs such that all the vertices are forced codewords. For example, the complete graphs and the complete bipartite graphs, where both independent sets of the partition have at least two vertices, are such extreme graphs.

\subsection{Results based on Sperner's Theorem}\label{secSper}
One of the fundamental results on locating-dominating codes by Slater~\cite{S:DomLocAcyclic} says that if $G$ is a graph with $n$ vertices and $\LD(G) = k$, then $n \leq k + 2^k - 1$. This result is based on the simple fact that using $k$ codewords at most $2^k-1$ distinct, non-empty $I$-sets can be formed. In what follows, we present an analogous result for self-locating-dominating and solid-locating-dominating codes. However, here it is not enough that all the $I$-sets are non-empty and unique, but we further require that none of the $I$-sets is included in another one. For this purpose, we present  Sperner's theorem, which considers the maximum number of subsets of a finite set such that none of the subsets is included in another subset. 
Sperner's theorem has originally been presented in~\cite{Ssuem}, and for more recent developments regarding the Sperner theory, we refer to~\cite{Kst}.
\begin{theorem}[Sperner's theorem~\cite{Ssuem}] \label{ThmSperner}
Let $N$ be a set of $k$ elements and let $\mathcal{F}$ be a family of subsets of $N$ such that no member of $\mathcal{F}$ is included in another member of $\mathcal{F}$, i.e., for all distinct $X,Y \in \mathcal{F}$ we have $X \setminus Y \neq \emptyset$. Then we have
\[
|\mathcal{F}| \leq \binom{k}{\left\lfloor \frac{k}{2} \right\rfloor} \text{.}
\]
Moreover, the equality holds if and only if $\mathcal{F} = \{ X \subseteq N \mid |X| = k/2\}$ when $n$ is even, and $\mathcal{F} = \{ X \subseteq N \mid |X| = (k-1)/2\}$ or $\mathcal{F} = \{ X \subseteq N \mid |X| = (k+1)/2\}$ when $n$ is odd.
\end{theorem}

A family of subsets satisfying the conditions of the previous theorem is called a \emph{Sperner family}. In the following theorem, we apply Sperner's theorem to obtain an upper bound on the order of a graph based on the number of codewords in a solid-locating-dominating (or self-locating-dominating) code.
\begin{theorem} \label{ThmUpperBoundOnOrder}
Let $G$ be a graph with $n$ vertices and $C$ be a solid-locating-dominating code in $G$ with $k$ codewords. Then we have the following upper bound on the order of $G$:
\[
n \leq k + \binom{k}{\left\lfloor \frac{k}{2} \right\rfloor} \text{.}
\]
\end{theorem}
\begin{proof}
Let $C$ be a solid-locating-dominating code in $G$ with $k$ codewords. By the definition, for any distinct $u, v \in V \setminus C$, we have $I(C;u) \setminus I(C;v) \neq \emptyset$. Therefore, the $I$-sets of non-codewords of $G$ form a Sperner family of subsets of $C$. Thus, by Sperner's theorem, we obtain that
\[
|V \setminus C| = n - k \leq \binom{k}{\left\lfloor \frac{k}{2} \right\rfloor} \text{.}
\]
Hence, the claim immediately follows.
\end{proof}

Observe that the previous theorem also holds for self-locating-dominating codes due to Corollary~\ref{CorollarySLdtoDLD}. Furthermore, the upper bound of the theorem can be attained even for self-locating-dominating codes as is shown in the following example.
\begin{example} \label{ExampleMaximalSperner}
Let $k$ be a positive integer and $\ell$ be an integer such that $\ell = \binom{k}{\left\lfloor k/2 \right\rfloor}$. Consider then a bipartite graph $G$ with the vertex set $U \cup V$, where $U = \{ u_1, u_2, \ldots, u_k \}$ and $V = \{ v_1, v_2, \ldots, v_{\ell} \}$. There are no edges within the sets $U$ and $V$, and the edges between the two sets are defined as follows. Let $\mathcal{F}$ be a (maximum) Sperner family of $U$ attaining the upper bound of Theorem~\ref{ThmSperner} with each subset of $\mathcal{F}$ having $\lfloor k/2 \rfloor$ elements. Recall that the number of subsets in $\mathcal{F}$ is $\ell$. Denoting the subsets of $\mathcal{F}$ by $F_1, F_2, \ldots, F_{\ell}$, we define the edges of each $v_i$ as follows: $v_i$ is adjacent to the vertices of $F_i$.

Now the code $C = U$ is self-locating-dominating in $G$. Indeed, the $I$-sets of the non-codewords in $V$ form a Sperner family and, hence, 
the characterization~(i) of Theorem~\ref{ThmCharacterizationSLDDLD} is satisfied. Thus, $C$ is a self-locating-dominating code in $G$ with $k$ codewords and $G$ is a graph with $k + \binom{k}{\left\lfloor k/2 \right\rfloor}$ vertices.
\end{example}

In the following immediate corollary of Theorem~\ref{ThmUpperBoundOnOrder}, we give a lower bound on the minimum size of solid-locating-dominating and self-locating-dominating codes based on the order of a graph. Notice also that the obtained lower bounds can be attained by the construction given in the previous example.
\begin{corollary}
Let $G$ be a graph with $n$ vertices and let $k$ be the smallest integer such that $n \leq k + \binom{k}{\left\lfloor k/2 \right\rfloor}$. Then we have
\[
\SLD(G) \geq \DLD(G) \geq k \text{.}
\]
\end{corollary}

\medskip

\subsection{Results using the Dilworth Number}\label{secDil}

In what follows, we are going to present some results on self-location-domination and solid-location-domination based on certain properties or parameters of graphs. For this purpose, we first present some definitions and notation. Let $u$ and $v$ be distinct vertices of $G$.
We say that $u$ and $v$ are \emph{false twins} if $N(u) = N(v)$ and that $u$ and $v$ are \emph{true twins} if $N[u] = N[v]$. Furthermore, we say that $u$ and $v$ are \emph{twins} if they are false or true twins. Then a graph is called \emph{twin-free} if there does not exist a pair of twin vertices.

The characterization of forced codewords regarding self-location-domination in Theorem~\ref{ThmForcedCodewordsChar} motivates us to recall the following definition from~\cite{FoHa1978}. For a graph $G$, the \emph{vicinal preorder} $\lesssim$  is defined on $V(G)$ as follows:
\[
x\lesssim y \text{ if and only if } N(x)\subseteq N[y].
\]
In other words, a vertex $x$ is a forced codeword if and only if there exists a vertex $y$ such that $x \lesssim y$ by Theorem~\ref{ThmForcedCodewordsChar}. It is easy to see that $\lesssim$ is in fact a preorder, that is, a reflexive and transitive relation. We use the following notation:
\begin{itemize}
\item $x \sim y$ for ($x \lesssim y$  and $y \lesssim x$),
\item $x < y$ for ($x \lesssim y$  and not $y \lesssim x$),
\end{itemize}
A \emph{chain} is a subset $B\subseteq V(G)$ such that for any two elements $x$ and $y$ of $B$, $x \lesssim y$ or $y\lesssim x$ must hold. An \emph{antichain} is a subset $A \subseteq V(G)$ such that for any $x,y\in  A$, $x\lesssim y$ implies $x = y$. A vertex $x$ is \emph{maximal} if there is no vertex $y$ satisfying $x<y$. The existence of at least a single maximal vertex in the vicinal preorder is guaranteed in every \emph{finite} graph.

\begin{lemma}
Let $G$ be a graph of order {$n \ge 2$}. Then the following
statements hold.
\begin{enumerate}
\item If $x,y\in V(G)$ are neighbours, then $x \sim y$ if and only if $x$ and $y$ are true twins. On the other hand if $x$ and $y$ are not neighbours, then $x \sim y$ if and only $x$ and $y$ are false twins.
\item  A vertex $x\in V(G)$ is a forced codeword if and only if there exists $y\neq x$ such that $x\lesssim y$.
\end{enumerate}
\end{lemma}

As a consequence, we obtain the following properties of extreme graphs, for the self-location-domination number.
\begin{corollary} \label{cor:maximalSLD}
Let $G$ be a graph of order {$n\ge 2$}. Then $\gamma^{SLD}(G)=n$
if and only if every maximal vertex in the vicinal preorder has a
twin.
\end{corollary}
\begin{proof}
Suppose that $\gamma^{SLD}(G)=n$, then every vertex of $G$ is a forced codeword. In particular, if $x$ is a maximal vertex in the vicinal preorder, then there exists $y\neq x$ such that $x\lesssim y$. By the maximality of $x$, we obtain that $x\sim y$ and therefore, $x$ and $y$ are twin vertices. Suppose now that every maximal vertex has a twin, so maximal vertices are forced codewords. Let $u\in V(G)$ be a non-maximal vertex. Consequently, there exists $v\in V(G)$ such that $u\lesssim v$ and $u$ is a forced codeword. Therefore, every vertex in $G$ is a forced codeword and $\gamma^{SLD}(G)=n$, as desired.
\end{proof}

Some graphs satisfying the conditions of the previous corollary are, for example, graphs with at least two vertices with full degree, that is, vertices which are connected to all other vertices.
By the previous corollary, we immediately obtain the following result.
\begin{corollary}
If $G$ is a twin-free graph of order {$n\ge 2$}, then we have
$\gamma^{SLD}(G)\leq n-1$.
\end{corollary}

In order to characterize graphs having the greatest solid-location-domination number, we will use the Dilworth number, whose definition we quote from~\cite{FoHa1978}. The \emph{Dilworth number} $\nabla(G)$ of a graph $G$ is the minimum number of chains of the vicinal preorder covering $V(G)$. According to the well-known theorem of Dilworth (see~\cite{Di1950}), $\nabla(G)$ is equal to the cardinality of the maximum size antichains in the vicinal preorder. In the following results, we describe the relationship between the Dilworth number and the solid-location-domination number.
\begin{lemma}
Let $G$ be a graph and $C$ be a solid-locating-dominating code. Then $V(G)\setminus C$ is an antichain of the vicinal preorder.
\end{lemma}
\begin{proof}
Let $x,y\in V(G)\setminus C$ and suppose that $x\lesssim y$. Hence $N(x)\subseteq N[y]$. Because $x,y\notin C$, we obtain that $I(x) = N[x]\cap C = N(x)\cap C \subseteq N[y]\cap C = I(y).$
Therefore, $x=y$ by the definition of a solid-locating-dominating code.
\end{proof}

Using the previous result, we obtain the following lower bound.
\begin{corollary}\label{cor:lowerbound}
Let $G$ be a graph with $n$ vertices. Then $n-\nabla (G)\leq \gamma^{DLD}(G)$.
\end{corollary}
\begin{proof}
Let $C$ be an optimal solid-locating-dominating code of $G$. Then the set $V(G)\setminus C$ is an antichain of the vicinal preorder of $G$ and, therefore,
$$n-\gamma^{DLD}(G)=\vert V(G)\vert -\vert C \vert =\vert V(G)\setminus C\vert \leq \nabla (G).$$
\end{proof}

This lower bound for the solid-location-domination number will allow us to characterize graphs where this parameter reaches its maximum value $n-1$ among graphs with at least one edge. Recall that a graph is a \emph{threshold graph}~\cite{ChHa1973} if it can be constructed from the empty graph by repeatedly adding either an isolated vertex or a universal vertex (sometimes also called a dominating vertex), i.e., a vertex adjacent to all the existing vertices. It is well known that the following statements are equivalent \cite{Treshold}:
\begin{itemize}
\item $G$ is a threshold graph,
\item $\nabla (G)=1$,
\item the vicinal preorder in $V(G)$ is total, that is, $V(G)$ is a chain of the vicinal preorder.
\end{itemize}

In the following proposition, we characterize all the graphs $G$ attaining the maximum solid-location-domination number of $n-1$ (when we have at least one edge in a graph).
\begin{proposition}
Let $G$ be a graph of order $n$ and size $m\geq1$. Then $\gamma^{DLD}(G)=n-1$ if and only if $G$ is a threshold graph.
\end{proposition}
\begin{proof}
Theorem \ref{ThmDLDn-1} gives that $\gamma^{DLD}(G)\leq n-1$. If $G$ is a threshold graph, then $\nabla (G)=1$ and $n-1=n-\nabla (G)\leq \gamma^{DLD}(G)$. Hence, $\gamma^{DLD}(G)=n-1$.

Suppose now that $\gamma^{DLD}(G)=n-1$ and let $x,y\in V(G)$ be such that $x\neq y$. We will show that $x\lesssim y$ or $y\lesssim x$. Denote $C=V(G)\setminus\{x,y\}$. Observe that $C$ is not a solid-locating-dominating code as $\gamma^{DLD}(G)=n-1$. {If $I(C;x) = I(x) =\emptyset$, then $N(x)\subseteq N[y]$ and $x\lesssim y$. Analogously $I(y)=\emptyset$ implies $y\lesssim x$. Assume now that $I(x)\neq \emptyset$ and $I(y)\neq \emptyset$. Because $C$ is not  solid-locating-dominating,} we obtain $I(x)\subseteq I(y)$ or $I(y)\subseteq I(x)$. We may assume without loss of generality that $I(x)\subseteq I(y)$. Now we have $N(x)\setminus \{y\} = N[x]\cap (V(G)\setminus \{x,y\})=I(x)\subseteq I(y)= N[y]\cap (V(G)\setminus \{x,y\})=N(y)\setminus \{x\}\subseteq N(y)$. Therefore $N(x)\subseteq N(y)\cup \{y\}$ or equivalently $x\lesssim y$. For every pair of vertices $x,y\in V(G)$, we have obtained that $x\lesssim y$ or $y\lesssim x$. This means that the vicinal preorder is total or equivalently that $G$ is a threshold graph.
\end{proof}

\subsection{Independent Sets and Complements}\label{secIndandComp}

In what follows, we present upper bounds on the self-location-domination and solid-location-domination numbers based on the independence number and the maximum degree of the graph. Recall that a set $S \subseteq V(G)$ is \emph{independent} in $G$ if no two vertices in $S$ are adjacent. Furthermore, the \textit{independence number} $\beta(G)$ of $G$ is the maximum size of an independent set in $G$. Moreover, a set $S$ is called \emph{$3$-distance-independent} if   we have $d(v,u)\geq3$ for each pair of vertices $v,u\in S$. We denote the maximal size of $3$-distance-independent set in $G$ with $\beta_2(G)$. Now we are ready to present the following theorem.
\begin{theorem} \label{DeltaBounds}  Let $G=(V,E)$ be a connected graph on $n\ge 2$ vertices with maximum degree $\Delta$.
\begin{itemize}
\item[(i)] Then we have
$$\DLD(G)\le n-\beta_2(G)\leq \left\lfloor n\left(1-\frac{1}{\Delta^2+1}\right) \right\rfloor.$$
\item[(ii)] If $G$ has  the additional property that $N(u)\not\subseteq N(v)$ for all distinct vertices $u,v\in V$,  then $$\DLD(G)\le n-\beta(G)\le \left\lfloor n\left(1-\frac{1}{\Delta+1}\right) \right\rfloor.$$
\item[(iii)] If $G$ has the property that $N(u)\not\subseteq N[v]$ for all distinct vertices $u,v\in V$, then
$$\SLD(G)\le n-\beta(G) \le \left\lfloor n\left(1-\frac{1}{\Delta+1}\right) \right\rfloor.$$
\end{itemize}
\end{theorem}
\begin{proof}
(i) Let us first consider a set $S\subseteq V$ which is obtained in the following way. Let $T_1=V.$ We choose first any $u_1\in T_1$ and then we set $T_2=T_1\setminus \cup_{v\in N(u_1)}N[v]$. Next we choose $u_2\in T_2$ and set $T_3=T_2\setminus \cup_{v\in N(u_2)}N[v].$ We continue this way by choosing $u_i\in T_i$ and defining $T_{i+1}=T_i\setminus \cup_{v\in N(u_i)}N[v]$ until $T_{i+1}=\emptyset$. Now we denote $S=\{u_1,u_2,\dots\}$ (this is a finite set).
Since the maximum degree equals $\Delta$, we know that on each round we remove from $T_i$ at most $\Delta^2+1$ vertices. Therefore,
$$
|S|\ge \frac{n}{\Delta^2+1}.
$$
Next we show that the code $C=V\setminus S$ is solid-locating-dominating. Observe that the  distance between two vertices  in $S$ (that is, the non-codewords in $V$) is at least three and hence, $S$ is $3$-distance-independent. Consequently, $I(u)\setminus I(v)=N(u)\neq \emptyset$ for any distinct non-codewords $u$ and $v$ (if $|S|=1$ we are immediately done). Thus,
$$
\DLD(G)\le |C|=n-|S|\le  n\left(1-\frac{1}{\Delta^2+1}\right).
$$

(ii)
In this case, let $S$ be an independent set in $G$ with $|S|=\beta(G).$ In what follows, we show that the code $C=V\setminus S$ is solid-locating-dominating. Let $u$ and $v$ be any non-codewords. If $d(u,v)\ge 3$, then clearly  $I(u)\setminus I(v)\neq \emptyset$ as above. Since $S$ is an independent set, it suffices to assume then that $d(u,v)=2$. We need to show that $I(u)\setminus I(v)\neq \emptyset.$ Notice that now $I(u)=N(u)$ and $I(v)=N(v)$. If $I(u)\setminus I(v)=\emptyset$, then  $N(u)\setminus N(v)= \emptyset$, which contradicts the property of the graph. Therefore, we have $$ \DLD(G)\le  n - \beta(G).$$ Furthermore, it is shown in \cite[page $278$]{BergeInd} that 
$$
|S|=\beta(G)\ge \frac{n}{\Delta+1}.
$$

(iii) Let $S$ be as in Case (ii) and $C=V\setminus S$. Take any $u\notin C$, that is, $u\in S$. Again $I(u)=N(u).$ We need to show that
$$
\bigcap_{c\in I(u)}N[c]=\{u\}.
$$
Assume to the contrary that the intersection contains another vertex, say $v \in V$, besides $u$. But this implies that $N(u)\subseteq N[v]$ which is not possible. Therefore, the assertion follows.
\end{proof}

The constraints $N(u)\not\subseteq N(v)$ and $N(u)\not\subseteq N[v]$ for all distinct vertices $u,v\in V$ have their purpose in the cases~(ii) and (iii) of the previous theorem. For example, if $G$ is a star on $n$ vertices and $v,v'$ are two distinct pendant vertices, then $N(v)\subseteq N(v')$. Moreover, we have $\beta(G)=\SLD(G)=\DLD(G)=n-1$ while $n-\beta(G)=1$. Observe also that the bound of (i) is now attained since we have $\beta_2(G)=1$ and $\DLD(G)= n-1 = n-\beta_2(G)$.


The bounds (ii) and (iii) of Theorem~\ref{DeltaBounds} can be attained, for example, when $G=C_t$ is a cycle on $t\geq5$ vertices. In these cases, we have $\beta(G)=\left\lfloor\frac{t}{2}\right\rfloor$. This implies that $\DLD(G)\leq \SLD(G)\leq\left\lceil\frac{t}{2}\right\rceil$ by the previous theorem. Moreover, let $C$ be a solid-locating-dominating code in a cycle $C_t$ where $t\geq5$ and let us consider four consecutive vertices $P=\{v_1,v_2,v_3,v_4\}$ of the cycle, where $v_iv_{i+1}\in E$ ($i \in \{1,2,3\}$). If $v_1$ is the only codeword in $P$, then $I(v_3)=\emptyset$. If $v_2$ is the only codeword in $P$, then $I(v_3)\subseteq I(v_1)$. The cases with $v_3$ and $v_4$ being the only codewords are symmetric. Hence, we have at least two codewords among every four consecutive vertices and there are $t$ different sets consisting of four consecutive vertices. On the other hand, each codeword belongs to four different sets of consecutive vertices. Therefore, by a double counting argument, we obtain that $4|C|\geq2t$ and hence, $\DLD(C_t)\geq \left\lceil\frac{t}{2}\right\rceil$. Thus, in conclusion, we have $\DLD(G) = \SLD(G)= \left\lceil\frac{t}{2}\right\rceil$ and the bounds~(ii) and (iii) are attained.

\medskip

We conclude the section by considering self-location-domination and solid-location-domination numbers in a graph and its complement. It has been shown in~\cite{HMPng} that in a graph and its complement the (regular) location-domination number always differs by at most one. In the following theorem, we show that a similar result also holds for solid-location-domination number. However, later in Remark~\ref{RemarkComplement}, it is shown that an analogous result \emph{does not} hold for self-location-domination number.
\begin{theorem}\label{komplementti}
Let $G$ be a graph on at least two vertices and $\overline{G}$ be its complement. We have $|\gamma^{DLD}(G)-\gamma^{DLD}(\overline{G})|\leq 1$ and the optimal codes are of different cardinality if and only if  $G$ is a complete or discrete graph.
\end{theorem}
\begin{proof}
Let $C$ be an optimal solid-locating-dominating code in $G$ and $v\in V(G)\setminus C$. Suppose that $I(G,C;w)\neq C$ for each vertex $w\in V(G)\setminus C$. Hence, $I(\overline{G},C;w)\neq \emptyset$ for each vertex $w\in V(G)\setminus C$. We have $I(\overline{G},C;v)=C\setminus I(G,C;v)$ and $I(G,C;v)=C\setminus I(\overline{G},C;v)$. If there exists a vertex $u\in V(G)\setminus C$ such that $I(\overline{G},C;u)\subseteq I(\overline{G},C;v)$, then $C\setminus I(G,C;u)\subseteq C\setminus I(G,C;v)$ and hence, $I(G,C;v)\subseteq I(G,C;u)$ which is a contradiction. Therefore, $C$ is also a solid-locating-dominating code for $\overline{G}$ and similarly we get that if $C'$ is a solid-locating-dominating code for $\overline{G}$ with no non-codewords adjacent to all codewords, then it is also a solid-locating-dominating code in $G$.

Let us then suppose that there is a vertex $v$ such that $I(G,C;v)=C$ and $v\in V(G)\setminus C$. We immediately notice that we then have only one non-codeword since if we had another non-codeword $u$, we would have $I(G,C;u)\subseteq I(G,C;v)$. Furthermore, in $\overline{G}$ we have $N[v]=\{v\}$, vertex $v$ is a codeword and thus, there are no vertices in $V(\overline{G})$ which would contain all codewords in their neighbourhoods. 
Hence, if we have $\gamma^{DLD}(\overline{G})\leq|V|-2$, then by the previous considerations we have $\gamma^{DLD}(G)\leq|V|-2$ which is a contradiction. Therefore, we may assume that $\gamma^{DLD}(\overline{G})\geq |V|-1$. Furthermore, the only graph for which we have $\gamma^{DLD}(\overline{G})=|V|$ is the discrete graph by Theorem \ref{ThmDLDn-1} and in that case $G$ is the complete graph.
\end{proof}

In the following remark, it is shown that an analogous result to the previous theorem does not hold for self-locating-dominating codes; in other words, the difference of the self-location-domination number of the graph and its complement can be arbitrarily large.
\begin{remark} \label{RemarkComplement}
Consider the graph $G = (V \cup U,E)$ of Example~\ref{ExampleMaximalSperner} with $k \geq 4$. Form a new graph $G' = (V \cup U, E')$ based on $G$ by adding edges between each pair of distinct vertices of $U$ (the subgraph graph induced by $U$ is now a clique with $k$ vertices). Then each vertex of $V$ is a forced codeword of a self-locating-dominating code by Theorem~\ref{ThmForcedCodewordsChar}. On the other hand, $V$ is a self-locating-dominating code in $G'$ by the characterization~(i) of Theorem~\ref{ThmCharacterizationSLDDLD}. Indeed, for any distinct vertices $u \in U$ and $w \in U \cup V$ there exists a vertex $v \in V$ such that $u \in N(v)$ and $w \notin N(v)$ (recall that the open neighbourhoods of the vertices in $V$ form a \emph{maximum} Sperner family). Thus, we have $\gamma^{SLD}(G') = |V| = \binom{k}{\lfloor k/2 \rfloor}$.

Consider then the complement graph $\overline{G'}$. Now the subgraph induced by $V$ is a clique and the intersections $N(v) \cap U$ of all the vertices $v \in V$ form a (maximum) Sperner family with $|N(v) \cap U| = \lceil k/2 \rceil$. Hence, as $V$ induces a clique, all the vertices of $U$ are forced codewords (by Theorem~\ref{ThmForcedCodewordsChar}). On the other hand, as in Example~\ref{ExampleMaximalSperner}, it can be shown that $U$ is a self-locating-dominating code in $\overline{G'}$. Thus, we have $\gamma^{SLD}(\overline{G'}) = |U| = k$. Therefore, in conclusion, we have shown that $|\gamma^{SLD}(G') - \gamma^{SLD}(\overline{G'})| = \binom{k}{\lfloor k/2 \rfloor} - k$.
\end{remark}

\section{Trees} \label{SectionTrees}

In this section, we study both the self-location-domination and the solid-location-domination number in \textit{trees}. 
We recall the following definition from \cite{BliChe05}. A \emph{$2$-dominating set} in a graph $G$ is a dominating set $S$ that dominates every vertex of $V\setminus S$ at least twice, i.e., $|I(S;u)| \geq 2$ for all $u \in V \setminus S$. The \emph{$2$-domination number} of $G$, which is the minimum cardinality of a $2$-dominating set of $G$, is denoted by $\gamma_2(G)$. In addition to the $2$-domination number, also the independence number $\beta(G)$ will play a role in this section. In general, both parameters are non-comparable, i.e., there are graphs where either of these values can be larger than the other one. However, we have $\beta (T)\leq \gamma_2(T)$ for every tree $T$ by \cite{BliChe05}. We will prove that, in the case of trees, self-locating dominating codes are precisely the $2$-dominating sets, and therefore the associated parameters also agree. We will also show that solid-location-domination number equals independence number in trees, in spite of associated sets \emph{are not agreeing} in general.

First of all, we focus on the relationship between self-locating-dominating codes and $2$-dom\-i\-nat\-ing sets. However, we require the concept of \textit{girth} of a graph $G$, that is, the length of shortest cycle in $G$. The graphs without cycles are considered to have an infinite girth.
\begin{lemma}\label{lem:SLD2domination}
Let $G$ be a graph.
\begin{itemize}
\item[(i)] Every self-locating-dominating code in $G$ is a $2$-dominating set.
\item[(ii)] If the girth of $G$ is at least $5$, then every $2$-dominating set of $G$ is a self-locating-dominating code.
\end{itemize}
\end{lemma}

\begin{proof}
(i) Let $C$ be a self-locating-dominating code. If there exists $u\in V(G)\setminus C$ such that $I(u)=\{v\}$, then $v\in \bigcap _{c\in I(u)} N[c]$, which is not possible. Hence, $C$ is a $2$-dominating set.

(ii) Let $G$ be a graph with girth at least $5$, $C$ be a $2$-dominating set and $u$ belong to $V(G)\setminus C$. By the hypothesis, there exist $c_1,c_2\in I(u)$, $c_1\neq c_2$, and since $G$ contains no cycles of length three, we know that $c_1$ is not a neighbour of $c_2$. Suppose that there exists $v\neq u$ such that $v\in \bigcap _{c\in I(u)} N[c]$, so $v\in N[c_1]\cap N[c_2]$. Moreover $v\neq c_1,c_2$, because $c_1c_2$ is not an edge of $G$. Again because $G$ has no triangles, we obtain that $u$  is not a neighbour of $v$ and therefore the vertex subset $\{u,c_1,v,c_2\}$ induces a $4$-cycle, a contradiction.
\end{proof}

The following corollary is an immediate consequence of the previous lemma and, in particular, it can be applied to every tree.
\begin{corollary}\label{cor:SLD2domination}
Let $G$ be a graph with girth at least $5$. Then $\gamma^{SLD}(G)=\gamma_2(G)$.
\end{corollary}

In what follows, we briefly discuss the previous requirement stating that the girth of the graph is at least $5$. Let us consider a graph $G=(V,E)$ where $V=K\cup P$, $K=\{v_1,v_2\}$, $P=\{u_1,\dots, u_p\}$, $p\geq2$, and we have $E=\{v_iu_j\mid v_i\in K, u_j\in P\}$. The graph $G$ has girth $4$, it has a $2$-dominating set $K$ and the unique self-locating-dominating code $C$ consists of whole $V$. Therefore, the requirement of girth at least $5$ is not only needed but removing it may cause an arbitrarily large difference between $\gamma_2(G)$ and $\SLD(G)$.

A particular case of trees are paths, where $2$-dominating numbers are known~\cite{MoKel12}. Therefore, using the above corollary, we obtain:
\begin{corollary}\label{cor:PathSLD} Let $n\geq2$ and $P_n$ be a path. Then we have $$\gamma^{SLD}(P_n)=\gamma_2(P_n)=\left\lceil \frac{n+1}{2}\right\rceil.$$
\end{corollary}

We now study the behaviour of the solid-location-domination number in trees, and we prove that it agrees with the independence number. We will need the following notation. A vertex in a tree $T$ is a \textit{leaf} if it is of degree one and a vertex is a \emph{support vertex} if there is at least one leaf in its neighbourhood. If $u$ is a support vertex, then $L_u$ will denote the set of leaves attached to it. In the following lemma we recall a result from \cite{BliChe05}.
\begin{lemma}[Lemma~3 of \cite{BliChe05}] \label{lem:independence_trees}
Let $T$ be a tree and let $u$ be a support vertex in $T$ such that $\vert N(u)\setminus L_u\vert =1$. If $T'=T-(L_u\cup \{u\})$, then $\beta(T')=\beta(T)-\vert L_u\vert$.
\end{lemma}

A similar result can be proved for the solid-location-domination number, as we show in the following lemma.
\begin{lemma}\label{lem:DLD_trees}
Let $T$ be a tree and let $u$ be a support vertex in $T$ such that $\vert N(u)\setminus L_u\vert =1$. If $T'=T-(L_u\cup \{u\})$, then $\gamma^{DLD}(T')=\gamma ^{DLD}(T)-\vert L_u\vert$.
\end{lemma}
\begin{proof}
Denote by $v$ the unique non-leaf neighbour of $u$ and let $C$ be an optimal solid-locating-dominating code in $T$. If $u\notin C$, then clearly $L_u\subseteq C$, to keep the domination. If $u,v\in C$ then, by minimality of $C$, there exists exactly one vertex in $L_u\setminus C$. And if $u\in C$ and $v\notin C$, then $L_u\subseteq C$, by definition of solid-locating-dominating code and, in this case, we define $C^*=(C\setminus \{u\})\cup \{v\}$, which can be straightforwardly shown to be an optimal solid-locating-dominating code in $T$ with $\vert C^*\vert =\vert C\vert$.

In all the cases, we have an optimal solid-locating-dominating code $C$ in $T$ such that $\vert C\cap (L_u\cup \{u\})\vert =\vert L_u\vert$. Note that in all cases $C'=C\setminus (L_u\cup \{u\})$ is a solid-locating-dominating code of $T'=T \setminus (L_u\cup \{u\})$. Hence, we have $$\gamma^{DLD}(T') \leq \vert C'\vert=\vert C\vert -\vert L_u\vert=\gamma ^{DLD}(T)-\vert L_u\vert.$$

Suppose that $\gamma^{DLD}(T')<\gamma ^{DLD}(T)-\vert L_u\vert$ and let $C''$ be a solid-locating-dominating code in $T'$ with $\vert C''\vert = \gamma ^{DLD}(T') <\gamma ^{DLD}(T)-\vert L_u\vert$. If $v\notin C''$, define $D=C''\cup L_u$. If $v\in C''$, pick a leaf $x\in L_u$ and define $D=C''\cup (L_u\setminus \{x\})\cup \{u\}$. In both cases, we obtain a solid-locating-dominating code $D$ of $T$ that satisfies $\vert D\vert =\vert C'' \vert +\vert L_u\vert <(\gamma ^{DLD}(T)-\vert L_u\vert) +\vert L_u\vert =\gamma ^{DLD}(T)$, which is a contradiction.
\end{proof}

We can now prove the following result that gives the desired equality between the solid-location-domi\-nation number and the independence number in trees.
\begin{proposition}\label{pro:DLDtrees}
Let $T$ be a tree. Then $\gamma ^{DLD}(T)=\beta (T)$.
\end{proposition}
\begin{proof}
If $T=K_{1,n-1}$ is a star with $n$ vertices, then it is clear that $\gamma ^{DLD}(T)=\beta (T)=n-1$. Assume now that $T$ is not a star. We proceed by induction on $n=\vert V(T)\vert$. The result is trivially true if $n=1$ or $n=2$. Let $n\geq 3$ be an integer and assume that the statement is true for trees with at most $n-1$ vertices.

 Since $T$ is not a star, there exists a support vertex $u$ such that $\vert N(u)\setminus L_u\vert =1$. By Lemma~\ref{lem:DLD_trees}, the tree $T'=T-(L_u\cup \{u\})$ satisfies $\gamma^{DLD}(T')=\gamma ^{DLD}(T)-\vert L_u\vert$. The inductive hypothesis gives that $\gamma ^{DLD}(T')=\beta (T')$ and, by Lemma~\ref{lem:independence_trees}, we know that $\beta(T')=\beta(T)-\vert L_u\vert$. Therefore $\gamma ^{DLD}(T)=\gamma^{DLD}(T')+\vert L_u\vert=\beta(T')+\vert L_u\vert=\beta(T)$, as desired.
\end{proof}

In the particular case of paths, independence number is known (\cite[Lemma 4]{PathInd}), and therefore:
\begin{corollary}\label{cor:PathDLD} We have $$\gamma^{DLD}(P_n)=\beta (P_n)=\left\lceil \frac{n}{2}\right\rceil.$$
\end{corollary}

In the above proposition, we have shown that $\gamma ^{DLD}(T)=\beta (T)$. Previously, the independence number has been extensively studied, and due to~\cite{BliCheFaMe07}, it is known that $\beta(G) \geq (n + \ell(G) - s(G))/2$. This lower bound immediately gives the following corollary. Observe that the lower bound can be attained by any path with even number of vertices (by the previous corollary). Moreover, this bound has been studied together with location-domination number, independence number and $2$-domination number in \cite{BliFaLo08}.
\begin{corollary}
Let $T$ be tree of order $n$, with $\ell(T)$ leaves and $s(T)$ support vertices. Then $\gamma^{DLD}(T)\geq (n + \ell(T) - s(T))/2$.
\end{corollary}

We have proved that in every tree, self-locating-dominating codes are exactly $2$-dominating sets and this gives the equality between associated parameters as shown in Corollary~\ref{cor:SLD2domination}. However, the equality between solid-location-domination number and independence number in trees as shown in Proposition~\ref{pro:DLDtrees}, does not imply any general relationship between minimum solid-locating-dominating codes and maximum independent sets.

The path with six vertices satisfies $\gamma ^{DLD}(P_6)=\beta (P_6)=3$. In Figure~\ref{fig:P6_ind}, we show all maximum independent sets in $P_6$ (squared vertices). In Figure~\ref{fig:P6_DLD}, we show all minimum solid-locating-dominating codes (squared vertices) in the same graph. In this case, none of the maximum independent sets is solid-locating-dominating and none of the minimum solid-locating-dominating codes is independent. However, occasionally the optimal solid-locating-dominating may also be an independent set like in the case of $P_3$ with the middle vertex as the non-codeword.

\begin{figure}[ht]
    \centering
    \begin{subfigure}{4.2cm}
    \includegraphics[width=1.150\textwidth]{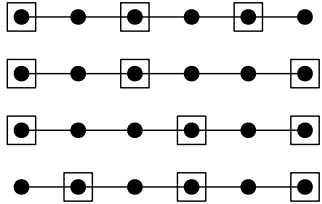}\vspace*{0.4cm}
    \caption{Maximum independent sets in $P_6$ are {denoted by the squared vertices}.}\label{fig:P6_ind}
    \end{subfigure}
    \hspace{3cm}
    \begin{subfigure}{4.2cm}
    \includegraphics[width=1.0\textwidth]{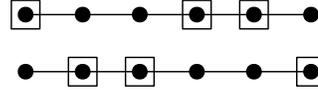}
    \caption{Minimum {solid-locating-dominating codes in $P_6$ are denoted by the squared vertices}.}\label{fig:P6_DLD}
    \end{subfigure}
    \caption{Maximum independent sets and minimum solid-locating-dominating codes are different.}%
    \label{fig:P6}%
\end{figure}


\section{Cartesian products and ladders} \label{SectionCartesian}

In this section, we consider self-location-domination and solid-location-domination in the \textit{Cartesian product of graphs}. The Cartesian product of graphs $G=(V(G),E(G))$ and $H=(V(H),E(H))$ is $G\square H=(V(G)\times V(H),E)$ where $(u,v)(u',v')\in E$ if and only if $u=u'$ and $vv'\in E(H)$ or $uu'\in E(G)$ and $v=v'$. We begin by presenting a theorem which gives lower and upper bounds for the self-location-domination and the solid-location-domination numbers for Cartesian products. Then we proceed by studying these numbers more closely in the Cartesian products $P_n \Box P_2$, where $P_k$ denotes a path with $k$ vertices. Using these results concerning $P_n \Box P_2$ and some other previously known ones for the Cartesian product of two complete graphs (see~\cite{JLLrntcld}), we are able to show that most of the obtained lower and upper bounds can be attained.

\begin{theorem} \label{ThmCartesianProduct}
We have
\begin{itemize}
\item[(i)] $\max \{ \SLD(G),\SLD(H)\}\le \SLD(G\Box H) \le \min\{|V(H)|\SLD(G),|V(G)|\SLD(H)\}$ and
\item[(ii)] $\max \{ \DLD(G),\DLD(H)\}\le \DLD(G\Box H) \le \min\{|V(H)|\DLD(G),|V(G)|\DLD(H)\}$.
\end{itemize}
\end{theorem}

\begin{proof} (i) Let us first show the upper bound on $\SLD(G\Box H).$ Without loss of generality, we may assume that $|V(H)|\SLD(G)\le |V(G)|\SLD(H)$. Let $C$ be a self-locating-dominating code in $G$ attaining $\SLD(G).$ Denote $D=\{(c,v)\mid c\in C, v\in V(H)\}$. Clearly, $|D|=|V(H)|\SLD(G).$ We will show that $D$ is self-locating-dominating in $G\Box H.$ We denote, for any $h\in V(H)$, the set $\{(u,h)\mid u\in V(G)\}$ by $L_h$ and we call it a \emph{layer}. Observe that for any vertices $(u_1, h)$ and $(u_2,h)$ ($u_1 \neq u_2$) in the same layer $L_h$ we have $N[(u_1, h)] \cap N[(u_2, h)] \subseteq L_h$. Let $x$ be any non-codeword in $G\Box H$, say $x=(u,h)\in L_h$ for some $h$ and  $u\in V(G)\setminus C$. Now the codewords in $I(G\Box H;x)$ all belong to $L_h$. Since $C$ is self-locating-dominating in $G$, we get (by the previous observation) that
$$\bigcap_{c\in I(G\Box H;x)}N[c]=\{x\}.$$

Next we consider the lower bound on $\SLD(G\Box H)$. Without loss of generality, say $\SLD(G)\ge \SLD(H).$ Let $D$ be a self-locating-dominating code in $G\Box H$ of cardinality $\SLD(G\Box H).$ Denote by $C (\subseteq V(G))$ the set which is obtained by collecting all the first coordinates from $D$. We claim that $C$ is self-locating-dominating in $G$. Let $u\in V(G)$ be a non-codeword with respect to $C.$ This implies that the vertices $(u,h)$ are non-codewords with respect to $D$ for all $h\in V(H)$ and hence, $I((u,h)) \subseteq L_h$. Since $D$ is self-locating-dominating, we know that for any layer $L_h$ the neighbourhoods of the codewords in $I(G\Box H,D;(u,h))$ intersect uniquely in $(u,h)$. Because the first coordinates of the codewords in $I(G\Box H,D;(u,h))$ belong to $I(G,C;u)$, we obtain
$$\bigcap_{c\in I(G,C;u)}N[c]=\{u\}.$$
Thus $C$ is self-locating-dominating and the claim follows by noticing that
$|C|\ge \SLD(G).$

\smallskip
(ii) We can again assume without loss of generality that  $|V(H)|\DLD(G)\le |V(G)|\DLD(H)$. Let $C$ be a solid-locating-dominating code in $G$ attaining $\DLD(G)$ and denote again  $D=\{(c,v)\mid c\in C, v\in V(H)\}$. In order to verify that $D$ is solid-locating-dominating in $G\Box H$, we show that $I(D;x)\setminus I(D;y)$ is non-empty for any distinct non-codewords $x,y\in V(G\Box H).$ Denote $x=(u,h)$ and $y=(v,h')$ for some $u,v\in V(G)$ and $h,h'\in V(H).$ If $h=h'$, then we are done, since $C$ is solid-locating-dominating. If $h\neq h'$, then the claim follows from the fact that $I(D;x)$ contains a codeword in the layer $L_h$ and $I(D;y)$ cannot contain that codeword (since $x$ and $y$ are non-codewords).

The proof of the lower bound is again similar ---  let $D$ be a solid-locating-dominating code in $G\Box H$ of cardinality $\DLD(G\Box H)$ and $C$ be a set obtained from its first coordinates. Now let $u\in V(G)$ and $v\in V(G)$ be non-codewords with respect to $C$. This implies that the vertices $(u,h)$ and $(v,h')$ are non-codewords with respect to $D$ for all $h,h'\in H.$ Since $D$ is solid-locating-dominating, we must have that $I(D;(u,h))\setminus I(D;(v,h))$ contains a codeword $(c,h)$ of $D$ in the layer $L_h$. Therefore, $c\in I(C;u)\setminus I(C;v)$ in $G$ and we are done.
\end{proof}

\begin{remark}
Due to Corollary~\ref{cor:maximalSLD}, for any complete graph $K_m$, we have $\SLD(K_m)=m$. Moreover, it has been shown in~\cite{JLLrntcld} that $\SLD(K_m\Box K_n)=m$ for $m\geq2n$. Therefore, the lower bound of Case~(i) of the previous theorem can be attained.
\end{remark}

In what follows, we focus on the self-location-domination and solid-location-domination numbers in the Cartesian product of paths $P_n$ and $P_2$. Using these results, we are able to show that the upper bounds in Cases~(i) and (ii) can be attained.

The Cartesian product of the paths $P_n$ and $P_2$ will be called the \emph{ladder (graph)} of length $n$. Furthermore, we use the following notation for the vertex sets of $P_n$ and $P_2$: $V(P_n)=\{ v_1,v_2,v_3,\dots v_n\}$ and $V(P_2)=\{1,2\}$, and so the vertex set of the Cartesian product $P_n\Box P_2$ is $V(P_n\Box P_2)=\{ (v_i,j)\colon 1\leq i\leq n,\ 1\leq j\leq 2\}$ (see Figure~\ref{fig:ladder}).
\begin{figure}[ht]
\begin{center}
\includegraphics[height=0.1\textheight]{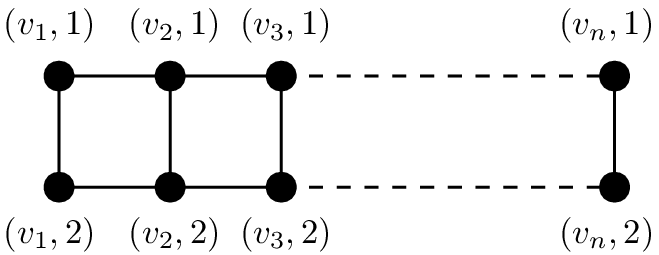}
\caption{The ladder $P_n\Box P_2$.}
\label{fig:ladder}
\end{center}
\end{figure}

The following notation will be useful in this section. Let $1\leq r\leq n$ be an integer. Now $P_r\Box P_2$ is the subgraph of $P_n\Box P_2$ induced by the vertex set $\{ (v_i,j)\colon 1\leq i\leq r,\ 1\leq j\leq 2\}$ (see Figure~\ref{fig:r_induced}), which is a ladder of length $r$. On the other hand, for an integer $0\leq s\leq n-1$, $P_{n-s}\Box P_2$ is the subgraph of $P_n\Box P_2$ induced by $\{ (v_i,j)\colon s+1\leq i\leq n,\ 1\leq j\leq 2\}$ (see Figure~\ref{fig:n-s_induced}), which is a ladder of length $n-s$.

\begin{figure}[ht]
    \centering
    \begin{subfigure}{4.5cm}
    \includegraphics[width=0.9\textwidth]{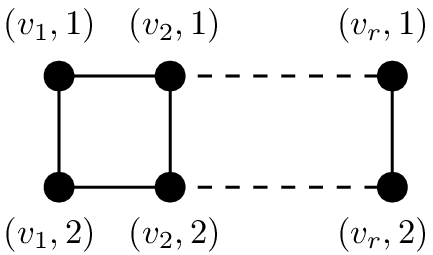}
    \caption{Subgraph $P_r\Box P_2$.}\label{fig:r_induced}
    \end{subfigure}
    \hspace{1.5cm}
    \begin{subfigure}{4.5cm}
    \includegraphics[width=1\textwidth]{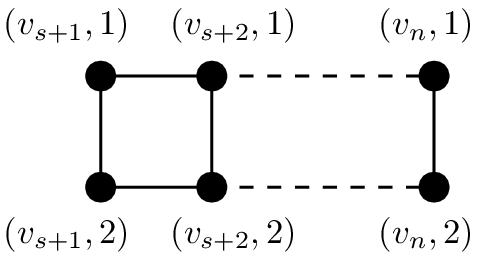}
    \caption{Subgraph $P_{n-s}\Box P_2$.}\label{fig:n-s_induced}
    \end{subfigure}
    \caption{Induced subgraphs in $P_n\Box P_2$.}%
    \label{fig:inducedsubgraphs}%
\end{figure}


We begin by computing the self-location-domination number of
ladders. To this end, we will use the relationship between
self-locating-dominating codes and $2$-dominating sets that we
showed in Lemma~\ref{lem:SLD2domination}. It is {known that}
{$\gamma_2(P_n\Box P_2)=n$} (see~\cite{MoKel12,ShaMaAl17}) for
$n\ge 2$ {and} {$\gamma_2(P_1\Box P_2)=2$}. In the next lemma,
we prove an additional property of optimal $2$-dominating sets that
will be useful to our purpose.
\begin{lemma}\label{lem:2dominating}
Let $n\geq 2$ be an integer and let $C\subseteq V(P_n\Box P_2)$ be a $2$-dominating set such that $\vert C\vert =n$. Then we have $\{ (v_1,1), (v_1,2)\}\nsubseteq C$.
\end{lemma}
\begin{proof}
There are exactly two $2$-dominating sets in $P_2\Box P_2$ with two vertices (see Figure~\ref{fig:case2}) and exactly two $2$-dominating sets in $P_3\Box P_2$ with three vertices (see Figure~\ref{fig:case3}). Therefore, the statement is clearly true for $n=2$ and $n=3$.

\begin{figure}[ht]
    \begin{center}
    \begin{subfigure}{4.5cm}
    \includegraphics[width=1.0\textwidth]{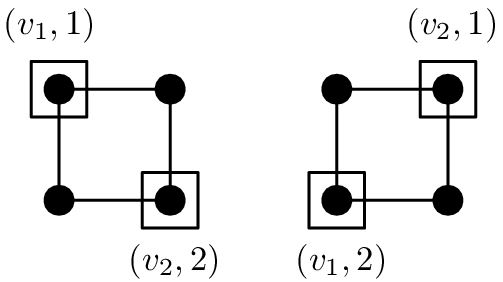}
    \caption{$2$-dominating sets with two vertices in $P_2\Box P_2$.}\label{fig:case2}
    \end{subfigure}
    \hspace{1cm}
    \begin{subfigure}{6.1cm}
    \includegraphics[width=1.15\textwidth]{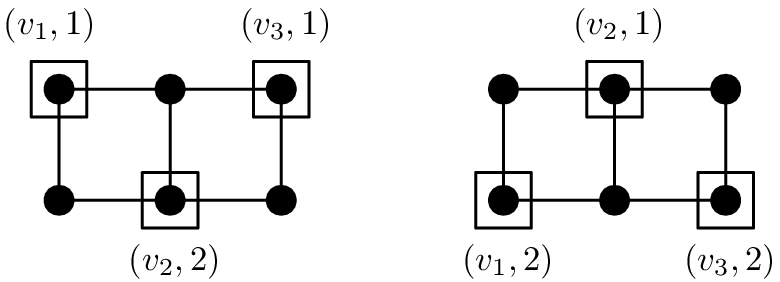}
    \caption{$2$-dominating sets with three vertices in $P_3\Box P_2$.}\label{fig:case3}
    \end{subfigure}
    \caption{Squared vertices are in the $2$-dominating set.}%
    \label{fig:cases2and3}%
    \end{center}
\end{figure}



We now proceed by induction on $n$. Assume the statement is true for $k<n$, $n\geq 4$, and let $C$ be a $2$-dominating set of $P_n\Box P_2$ such that $\vert C\vert =n$. Suppose to the contrary that $\{ (v_1,1), (v_1,2)\}\subseteq C$, then, because $C$ has $n$ elements, there exists $i\in \{2,\dots ,n\}$ such that $(v_i,1), (v_i,2)\notin C$. Note that  $i<n$, because $C$ is $2$-dominating. Now $\{(v_{i-1},1), (v_{i-1},2), (v_{i+1},1), (v_{i+1},2)\}\subseteq C$ to keep the $2$-domination.

Consider the induced subgraphs $G_1=P_{i-1}\Box P_2$ and $G_2=P_{n-i}\Box P_2$. It is clear that $C\cap V(G_1)$ and $C\cap V(G_2)$ are $2$-dominating sets in $G_1$ and $G_2$, respectively. Moreover, they satisfy $(v_{i-1},1),(v_{i-1},2)\in C\cap V(G_1)$ and $(v_{i+1},1),(v_{i+1},2)\in C\cap V(G_2)$.

Suppose that $i-1\geq 2$ and $n-i\geq 2$. By the inductive hypothesis $\vert C\cap V(G_1)\vert \geq (i-1)+1=i$ and $\vert C\cap V(G_2)\vert \geq (n-i)+1$. Therefore, $\vert C\vert\geq n+1$, which is a contradiction. Assume now that $i-1=1$ and $n-i=n-2$. In this case $\vert C\cap V(G_1)\vert =2$ and, by the inductive hypothesis, $\vert C\cap V(G_2)\vert \geq (n-2)+1=n-1$. Again $\vert C\vert\geq n+1$,  a contradiction. The remaining case, $i-1=n-2$ and $n-i=1$, is similar to the previous one. Therefore, $\{ (v_1,1), (v_1,2)\}\nsubseteq C$ as desired.
\end{proof}

This property gives that self-locating-dominating codes of ladders
$P_n\Box P_2$ are non-optimal $2$-dom\-i\-nat\-ing sets {for
$n\ge 2$}.
\begin{lemma}\label{lem:lowerbound}
Let $C$ be a self-locating-dominating code in $P_n\Box P_2$ with $n\ge 2$. Then $(v_1,1), (v_1,2),\linebreak (v_n,1), (v_n,2)\in C$ and $\vert C\vert \geq n+1$.
\end{lemma}
\begin{proof}
By Lemma~\ref{lem:SLD2domination}, $C$ is a $2$-dominating set in $P_n\Box P_2$. Hence, we have $\vert C\vert \geq n$. Suppose that $(v_1,1)\notin C$. Now $(v_1,2),(v_2,1)\in C$ and $(v_2,2)\in \bigcap_{c\in I((v_1,1))} N[c]$, which is not possible for a self-locating-dominating code. So $(v_1,1)\in C$ and analogously $(v_1,2),(v_n,1),(v_n,2)\in C$. Using Lemma~\ref{lem:2dominating}, we obtain that $\vert C\vert \geq n+1$.
\end{proof}

We can now compute the exact self-location-domination numbers of ladders.
\begin{theorem}
Let $n\geq 2$ be an integer. Then
$$\gamma^{SLD}(P_n\Box P_2)=
\left\{
  \begin{array}{ll}
    n+1 & \hbox{if $n$ is odd;} \\
    n+2 & \hbox{if $n$ is even.}
  \end{array}
\right.
$$
\end{theorem}
\begin{proof}
If $n=2k+1$, $k\geq 1$, then the set $\{(v_{2i+1}, 1), (v_{2i+1}, 2)\colon 0\leq i\leq k\}$ (see Figure~\ref{fig:oddladder}) is a self-locating-dominating code with $2(k+1)=n+1$ vertices. Thus, Lemma~\ref{lem:lowerbound} gives $\gamma^{SLD}(P_{2k+1}\Box P_2)=(2k+1)+1$.

\begin{figure}[ht]
\begin{center}
\includegraphics[height=0.1\textheight]{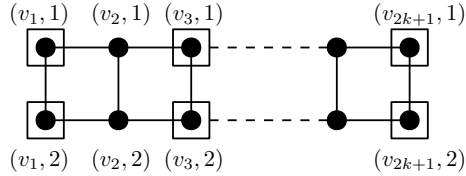}
\caption{Squared vertices form an optimal self-locating-dominating code in $P_{2k+1} \square P_2$.}
\label{fig:oddladder}
\end{center}
\end{figure}

Assume now that $n=2k$, $k\geq 1$. We will prove that $\gamma^{SLD}(P_{2k}\Box P_2)\geq 2k+2$, by induction on $k$. Clearly, $\gamma^{SLD}(P_2\Box P_2)=4$ (by the proof of Lemma~\ref{lem:lowerbound}). Assume that the statement is true for $r<k$, $k\geq 2$, and let $C\subseteq V(P_{2k}\Box P_2)$ be a self-locating-dominating code. Suppose that $\{(v_i,1),(v_i,2)\}\cap C\neq \emptyset$ for every $i\in \{1,2, \dots , 2k\}$. By Lemma~\ref{lem:lowerbound}, $(v_1,1),(v_1,2),(v_{2k},1),$ $(v_{2k},2)\in C$. So $\vert C\vert \geq 2k+2$.

Suppose now that there exists $i\in \{2,\dots , 2k-1\}$ such that
$(v_i,1),(v_i,2)\notin C$. Since $C$ is also a $2$-dominating set,
we obtain that $\{(v_{i-1},1), (v_{i-1},2), (v_{i+1},1),
(v_{i+1},2)\}\subseteq C$. Consider the induced subgraphs
$G_1=P_{i-1}\Box P_2$ and $G_2=P_{2k-i}\Box P_2$, with
self-locating-dominating codes $C\cap V(G_1)$ and $C\cap V(G_2)$,
respectively. Note that $(i-1)+(2k-i)$ is odd. {Let us assume
that $i-1$ is odd and $2k-i$ is even (the other case is analogous).
For $i-1$  odd, we know that $\vert C\cap V(G_1)\vert \geq
(i-1)+1$ (this holds also for $i=2$) and when $2k-i$ is an even number,
by the inductive hypothesis, $\vert C\cap V(G_2)\vert \geq
(2k-i)+2$.} This gives $\vert C\vert \geq 2k+2$.

Finally the set $\{(v_{2i+1}, 1), (v_{2i+1}, 2)\colon 0\leq i\leq k-1\}\cup \{ (v_{2k},1),(v_{2k},2)\}$ (see Figure~\ref{fig:evenladder}) is a self-locating-dominating code of $P_{2k}\Box P_2$ with $2k+2$ vertices and therefore $\gamma^{SLD}(P_{2k}\Box P_2)=2k+2$.
\end{proof}

\begin{figure}[ht]
\begin{center}
\includegraphics[height=0.1\textheight]{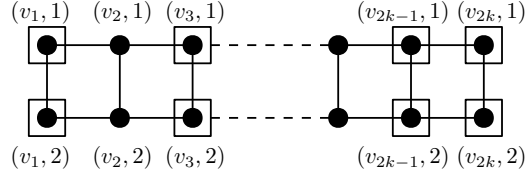}
\caption{Squared vertices form an optimal self-locating-dominating code in $P_{2k} \square P_2$.}
\label{fig:evenladder}
\end{center}
\end{figure}

\begin{remark}
Recall that we have $\SLD(P_n) = \gamma_2(P_n) = \lceil (n+1)/2 \rceil$ by Corollary~\ref{cor:PathSLD}. Notice then that, whether $n$ is even, with $n=2k$, or odd, with $n=2k+1$, we have shown that $\gamma^{SLD}(P_n\Box P_2)=2k+2=2(k+1)=\vert V(P_2)\vert \gamma^{SLD}(P_n)$. Therefore, the upper bound of Case~(i) of Theorem~\ref{ThmCartesianProduct} is attained for both even and odd $n$.
\end{remark}


We now focus on the solid-location-domination number of ladders.
\begin{proposition}\label{pro:solidladder}
Let $n \geq 1$ be an integer and $C\subseteq V(P_n\Box P_2)$ be a solid-locating-dominating code. Then
$\vert C\vert \geq n.$
\end{proposition}
\begin{proof}
The statement is clearly true for $n=1$ and $n=2$. Assume now that $n\geq 3$ and the claim is true for every $k<n$. Let $C\subseteq V(P_n\Box P_2)$ be a solid-locating-dominating code and suppose to the contrary that $\vert C\vert< n$. Then there exists $i\in \{1,2,\dots, n\}$ such that $(v_i,1),(v_i,2)\notin C$.

\begin{figure}[ht]
\begin{center}
\includegraphics[height=0.12\textheight]{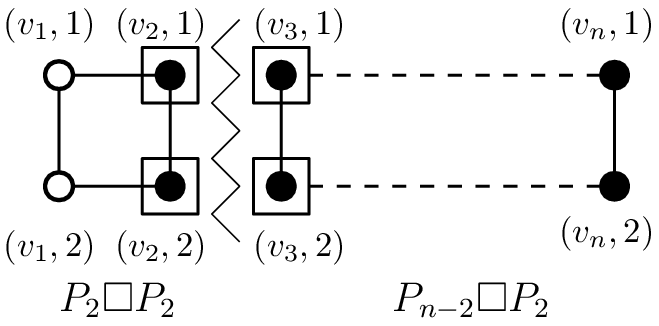}
\caption{Squared vertices are in $C$ and white vertices are not in $C$.}
\label{fig:solidladder1}
\end{center}
\end{figure}

Firstly suppose that $(v_1,1),(v_1,2)\notin C$, then using the fact that $C$ is a solid-locating-dominating code we obtain that $(v_2,1),(v_2,2),(v_3,1),(v_3,2)\in C$. Consider the induced subgraphs $G_1=P_2\Box P_2$ and $G_2=P_{n-2}\Box P_2$. Clearly both $C\cap V(G_1)$ and $C\cap V(G_2)$ are solid-locating-dominating codes in $G_1$ and $G_2$ respectively (see Figure~\ref{fig:solidladder1}). Moreover, $\vert C\cap V(G_1)\vert =2$ and by the inductive hypothesis $\vert C\cap V(G_2)\vert \geq n-2$. Hence, we have $\vert C\vert \geq n$ (a contradiction). Therefore, we may assume that $1<i$ and, with the same reasoning, that $i<n$.

Assume now that $(v_{i+1},1), (v_{i+1},2)\notin C$. Consequently, the vertices $(v_{i-1},1)$, $(v_{i-1},2)$, $(v_{i+2},1)$ and $(v_{i+2},2)$ belong to $C$.
 Consider the induced subgraphs $G_1=P_i\Box P_2$ and $G_2=P_{n-i}\Box P_2$ with solid-locating-dominating codes $C\cap V(G_1)$ and $C\cap V(G_2)$, respectively (see Figure~\ref{fig:solidladder2}). By the inductive hypothesis $\vert C\cap V(G_1)\vert \geq i$ and $\vert C\cap V(G_2)\vert \geq n-i$, so $\vert C\vert \geq n$, which is a contradiction. A similar argument can be used if $(v_{i-1},1),(v_{i-1},2)\notin C$.

\begin{figure}[ht]
    \begin{center}
    \begin{subfigure}{4.5cm}
    \includegraphics[width=1\textwidth]{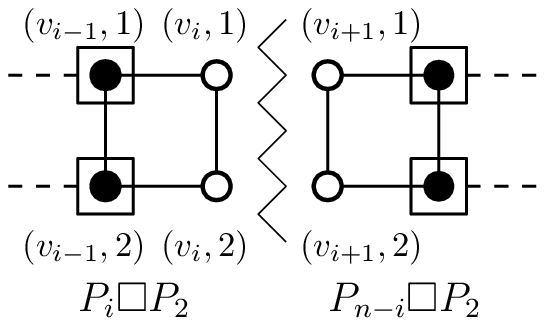}
    \caption{}\label{fig:solidladder2}
    \end{subfigure}
    \hspace{1cm}
    \begin{subfigure}{4.5cm}
    \includegraphics[width=1\textwidth]{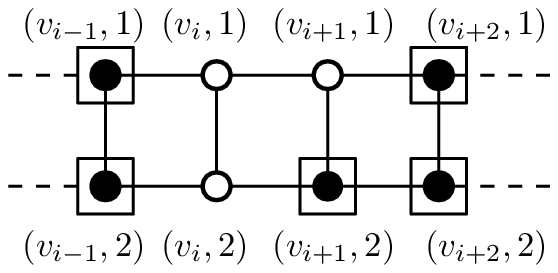}%
    \caption{}\label{fig:solidladder3}%
    \end{subfigure}
    \caption{Squared vertices are in $C$ and white vertices are not in $C$.}%
    \end{center}
\end{figure}




Suppose that $(v_{i+1},1)\notin C$ and $(v_{i+1},2)\in C$. Clearly, $i+2\leq n$ and $(v_{i+2},1)\in C$ because otherwise $I((v_{i+1},1))\subseteq I((v_i,2))$. Moreover, $(v_{i+2},2)\in C$ because otherwise $I((v_{i+1},1))\subseteq I((v_{i+2},2))$. Furthermore, $(v_{i-1},1)\in C$ since $I((v_{i},1))\neq\emptyset$ and $(v_{i-1},2)\in C$ as otherwise $I((v_i,1)) \subseteq I((v_{i-1},2))$. Hence, there are two pairs of codewords $\{(v_{i-1},1),(v_{i-1},2)\}$ and $\{(v_{i+2},1),(v_{i+2},2)\}$ such that the pair of non-codewords $\{(v_{i},1),(v_{i },2)\}$ is between them (see Figure~\ref{fig:solidladder3}).


We have stated that there cannot be a pair of non-codewords at the beginning or the end of the ladder. Furthermore, we may assume that there are no consecutive pairs of non-codewords and, if there is a pair of non-codewords $(v_{i},1),(v_{i },2)$, then there are two pairs of codewords $\{(v_{j},1),(v_{j },2)\}$ and $\{(v_{j'},1),(v_{j' },2)\}$, $j<i<j'$, such that if for some $i'$, $j<i'<j'$, we have $(v_{i'},1),(v_{i' },2)\not\in C$, then $i'=i$. Hence, the number of vertex pairs such that $\{(v_{j},1),(v_{j },2)\}\subseteq C$ is greater or equal to the number of vertex pairs $(v_{i},1),(v_{i },2)\not\in C$ and thus, we have $|C|\geq n$.
\end{proof}

We can finally determine the exact solid-location-domination numbers of all ladders.
\begin{corollary}
If $n\geq 1$ is an integer, then $\gamma^{DLD}(P_n\Box P_2)=n$.
\end{corollary}
\begin{proof}
The set $C=\{ (v_i,1)\colon 1\leq i\leq n\}$ is a solid-locating-dominating code in $P_n\Box P_2$ with $n$ elements since $I((v_i,2))=\{(v_i,1)\}$. So, $\gamma^{DLD}(P_n\Box P_2)\leq n$. The reverse inequality comes from Proposition~\ref{pro:solidladder}.
\end{proof}

\begin{remark}
Recall that we have $\DLD(P_n) = \beta(P_n) = \lceil n/2 \rceil$ by Corollary~\ref{cor:PathDLD}. Notice that, if $n$ is an even number, then $\gamma^{DLD}(P_n\Box P_2)=2 \cdot \tfrac{n}{2}=\vert V(P_2)\vert \gamma^{DLD}(P_n)$. Therefore, the upper bound shown in Case~(ii) of Theorem~\ref{ThmCartesianProduct} is attained in this case.
\end{remark}

\section{Realization theorems} \label{SectionRealization}

In this section, we consider location-domination, self-location-domination and solid-location-dom\-i\-nat\-ion numbers; in particular, we study what are the values the location-domination number can simultaneously have with the self-location-domination or the solid-location-domination number in a graph. Similar types of questions have been previously studied in~\cite{CHMPPbec} regarding various values such as domination number, location-domination number and metric dimension. In the following theorem, we characterize which values of location-domination and self-location-domination numbers can be simultaneously achieved in a graph.
\begin{theorem}\label{realization LD SLD}
Let $a$ and $b$ be positive integers. Then there exists a graph $G$ such that $a=\gamma^{LD}(G)$ and $b=\gamma^{SLD}(G)$ if and only if we have $$0\leq  b-a\leq 2^a-1.$$
\end{theorem}
\begin{proof}We cannot have $a>b$ since each self-locating-dominating code is also locating-dominating. We also cannot have $b> a+2^a-1$ since we can have at most $a+2^a-1$ vertices in a graph with locating-dominating code of cardinality $a$ by \cite{S:DomandRef}. Hence, $b-a$ is in the claimed interval. Based on the difference $b-a$, the proof divides into the following cases: (i) $a=b$, (ii) $b-a=1$, (iii) $a=2$ and $b=4$ or $b=5$, (iv) $ 2\leq b-a\leq 2^{a}-2$ and $a\geq3$, and (v) the extremal case $a\geq 3$ and $b-a=2^a-1$.

(i) Let us first study the case $a=b$. We can now consider the discrete graph $G$, that is, the graph with no edges, of order $a$ and we have $\gamma^{SLD}(G)= \gamma^{LD}(G)$.

(ii) Let us then study the case $b=a+1$. We can consider graph $G$ of order $b$ with one edge and we have $\gamma^{SLD}(G)= \gamma^{LD}(G)+1$.

(iii) Let us then study the case  $a=2$ and $b=4$. We immediately notice that these numbers are realized in the graph of Figure~\ref{fig:realization_2_4}. The case $a=2$ and $b=5$ is given in the graph of Figure~\ref{fig:realization_2_5}.

\begin{figure}[ht]
    \begin{center}
    \begin{subfigure}{4.3cm}
    \includegraphics[width=1\textwidth]{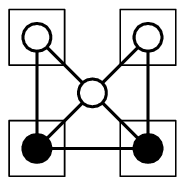}
    \caption{$\gamma^{LD}(G)=2, \gamma^{SLD}(G)=4$.}\label{fig:realization_2_4}
    \end{subfigure}
    \hspace{1cm}
    \begin{subfigure}{4.4cm}
    \includegraphics[width=1\textwidth]{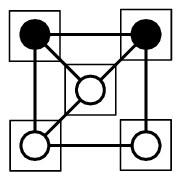}
    \caption{$\gamma^{LD}(G)=2, \gamma^{SLD}(G)=5$.}\label{fig:realization_2_5}
    \end{subfigure}
    \caption{Black vertices are in an optimal locating-dominating code and squared vertices are in an optimal self-locating-dominating code.}%
   \label{fig:realization_2}
    \end{center}
\end{figure}


(iv) Let us then study the case $ 2\leq  b-a\leq 2^{a}-2$ and $a\geq3$. There is an integer $k\in \Z$ such that $2^{k-1}-2< b-a\leq2^{k}-2$. Notice that $2\leq k\leq a$. Let us have $K=\{v_i\mid1\leq i\leq k\}$, $K'=\{v_i\mid k+1\leq i\leq a\}$, $P=\{u_i\mid1\leq i\leq b+1-a\}$ and $V=K\cup K'\cup P$. We have $|V|=b+1$. Let us connect $u_1$ to each vertex in $V$ and each vertex $u_i$, $2\leq i\leq k+1$ to a single vertex $v_j$, $j=i-1$. Indeed, the latter edges are possible since for any $k \geq 3$ we have $|P|\geq 2^{k-1}-1+1\geq k+1$ and if $k = 2$, then $b-a = 2$ and $|P| \geq 3$. 
Let us further connect each other vertex in $P$ to some proper non-empty subset of vertices in $K$ in such a manner that no two vertices in $P$ have the same neighbourhood in $K$. This choice of non-identical neighbourhoods is possible since we have $|P|\leq 2^{|K|}-1 = 2^k-1$. A graph $G$ with $k=4$ and $|V|=16$ is shown in the Figure \ref{SLD15LD6example}.

Let us first consider self-location-domination in $G$. There are $b$ forced codewords in a self-locating-dominating code since we have $N(v)\subseteq N[u_1]$ for each $v\in V$. Hence, we have $\gamma^{SLD}(G)\geq b$ and $V\setminus\{u_1\}$ is a self-locating-dominating code since $I(u_2)\cap I(u_3)=\{u_1\}$ and thus, $\gamma^{SLD}(G)= b$.

Let us then consider location-domination in $G$. We can choose each vertex in $K\cup K'$ as a codeword and have a locating-dominating code of size $a$. Let us show that $\gamma^{LD}(G)\geq a$.  We have $N[u_{j+1}]=\{u_1,u_{j+1},v_j\}$ for each $1\leq j\leq k$. Hence, if $v_j$ and $v_j'$ in $K$ are non-codewords, then at least two of vertices $u_{j+1}$, $u_{j'+1}$ and $u_1$ are codewords. Furthermore, by the same idea, if we have $t$ non-codewords in $K$, then there are at least $t$ codewords in $P$. Hence, we have at least $k$ codewords in $K\cup P$. Since we have $N(v_i)=\{u_1\}$ for each vertex $v_i$ in $K'$, we have at least $|K'|$ codewords in $K'\cup \{u_1\}$. If $u_1$ is a non-codeword, then it is immediate that we have $|K|+|K'|=a$ codewords in $G$. Hence, we may assume that $u_1$ is a codeword. If all the vertices $v_{k+1}, v_{k+2}, \ldots, v_{a}$ are codewords, then we are again immediately done. Moreover, by the previous observations at most one of the vertices can be a non-codeword. Hence, we may assume that there exists a unique non-codeword 
$v_y$, $y\geq k+1$. Now we have $I(v_y)=\{u_1\}$. Therefore, for any $1 \leq j\leq k$ at least one of $v_j$ and $u_{j+1}$ is a codeword as otherwise $I(u_{j+1}) = \{u_1\} = I(v_y)$ (a contradiction). Thus, there exist $|K| = k$ codewords in $K\cup P\setminus \{u_1\}$. Hence, in all the cases, we have $\gamma^{LD}(G)\geq a$.




(v) Let us finally study the extremal case $a\geq3$ and $b-a=2^a-1$. Let us consider graph $G=(V,E)$. Let us have  $K=\{v_i\mid1\leq i\leq a\}$, $P=\{u_i\mid1\leq i\leq b-a\}$ and $V=K\cup P$. Let us connect
\begin{enumerate}
\item $v_1$ to $v_i$ for each $2\leq i\leq a$
\item $u_1$ to $v_i$ for each $1\leq i\leq a$
\item $u_i$, $2\leq i\leq a$, to $v_1$ and $v_{i}$
\item $u_i$, $a+1\leq i\leq b-a$, to some non-empty subset of vertices of $K$ in such a manner that no two vertices of $P$ have the same neighbourhood in $K$. This is possible since $|P|\leq 2^{|K|}-1$.
\item $u_1$ to $u_i$, $2\leq i\leq b-a$, if $u_i\in N(v_1)$
\item $u_i$, $2\leq i\leq a$, to $u_j$, $a+1\leq j\leq b-a$, if $u_j\in N(v_i)$.
\end{enumerate}

Since we have $2^a+a-1$ vertices in the graph, we have $\gamma^{LD}(G)\geq a$. On the other hand, we can choose $K$ as an optimal locating-dominating code, since if $u,u'\in P$, then $N(u)\cap K\neq N(u')\cap K$ and $|K|=a$. In order to prove that $\gamma^{SLD}(G) = b = |V|$, we have to show that each vertex of the graph is a forced codeword. By Theorem \ref{ThmForcedCodewordsChar}, it suffices to show that for each vertex $v\in V$ there exists another vertex $v'$ such that $N(v)\subseteq N[v']$. Indeed, it is straightforward to verify that $N[v_i]=N[u_i]$ for $1\leq i\leq a$ and $N(u_i)\subseteq N[u_1]$ for $a+1\leq i\leq b-a$. 
Thus, the claim follows.
\end{proof}

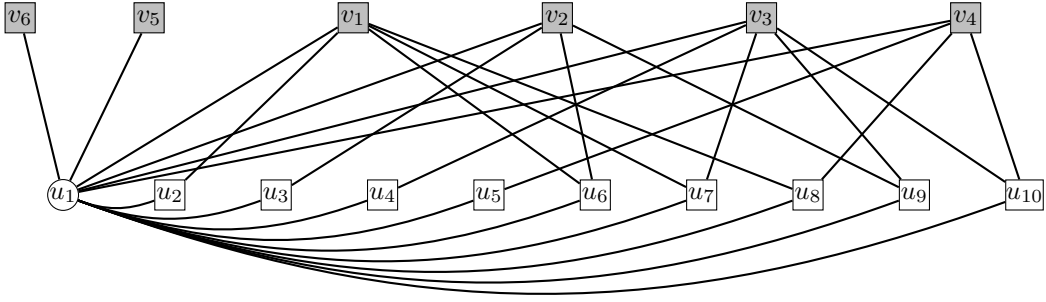
\begin{figure}
\centering
\begin{tikzpicture}
    \node[main node, minimum size=0.4cm] (1)[rectangle,draw][fill=lightgray]  {$v_1$};

    \node[main node, minimum size=0.4cm] (2)[rectangle,draw] [right = 2.3cm of 1][fill=lightgray]    {$v_2$};
    \node[main node, minimum size=0.4cm] (3) [rectangle,draw][right = 2.3cm of 2][fill=lightgray]  {$v_3$};
    \node[main node, minimum size=0.4cm] (4) [rectangle,draw][right = 2.3cm of 3][fill=lightgray]   {$v_4$};
    \node[main node, minimum size=0.4cm] (6) [below left = 2cm and 3.5cm of 1][fill=white]  {$u_1$};
    \node[main node, minimum size=0.4cm] (7)[rectangle,draw] [right = 1.0cm of 6][fill=white]   {$u_2$};
    \node[main node, minimum size=0.4cm] (8)[rectangle,draw] [right = 1.0cm of 7][fill=white]   {$u_3$};
    \node[main node, minimum size=0.4cm] (9) [rectangle,draw][right = 1.0cm of 8][fill=white]   {$u_4$};
    \node[main node, minimum size=0.4cm] (10)[rectangle,draw] [right = 1.0cm of 9][fill=white]   {$u_5$};
    \node[main node, minimum size=0.4cm] (11)[rectangle,draw] [right = 1.0cm of 10][fill=white]           {$u_6$};
    \node[main node, minimum size=0.4cm] (12)[rectangle,draw] [right = 1.0cm of 11][fill=white]   {$u_7$};
    \node[main node, minimum size=0.4cm] (13)[rectangle,draw] [right = 1.0cm of 12][fill=white]   {$u_8$};
    \node[main node, minimum size=0.4cm] (14)[rectangle,draw] [right = 1.0cm of 13][fill=white]   {$u_9$};
    \node[main node, minimum size=0.4cm] (15)[rectangle,draw] [right = 1.0cm of 14][fill=white]   {$u_{10}$};
    \node[main node, minimum size=0.4cm] (16)[rectangle,draw] [left = 2.3cm of 1][fill=lightgray]   {$v_5$};
        \node[main node, minimum size=0.4cm] (17)[rectangle,draw] [left = 1.3cm of 16][fill=lightgray]   {$v_6$};

    \path[draw,thick]
    (6) edge node {} (1)
    (6) edge node {} (2)
    (6) edge node {} (3)
    (6) edge node {} (4)
    (6) edge node {} (17)
    (6) edge node {} (16)
    (7) edge node {} (1)
    (8) edge node {} (2)
    (9) edge node {} (3)
    (10) edge node {} (4)
    (11) edge node {} (1)
    (11) edge node {} (2)
    (12) edge node {} (1)
    (12) edge node {} (3)
    (13) edge node {} (1)
    (13) edge node {} (4)
    (14) edge node {} (2)
    (14) edge node {} (3)
    (15) edge node {} (3)
    (15) edge node {} (4)


    (6) edge [bend right=20, looseness=1] node {} (7)
    (6) edge [bend right=20, looseness=1] node {} (8)
    (6) edge [bend right=20, looseness=1] node {} (9)
    (6) edge [bend right=20, looseness=1] node {} (10)
    (6) edge [bend right=20, looseness=1] node {} (11)
    (6) edge [bend right=20, looseness=1] node {} (12)
    (6) edge [bend right=20, looseness=1] node {} (13)
    (6) edge [bend right=20, looseness=1] node {} (14)
    (6) edge [bend right=20, looseness=1] node {} (15)

    ;
\end{tikzpicture}
\caption{Graph $G$ with $k=4$, $\gamma^{LD}(G)=6$ and $\gamma^{SLD}(G)=15$. Squared vertices are codewords in an optimal self-locating-dominating code and gray vertices are codewords in an optimal locating-dominating code.}\label{SLD15LD6example}
\end{figure}

In the following theorem, we proceed by characterizing which values of location-domination and solid-location-domination numbers can be simultaneously achieved in a graph.
\begin{theorem}
Let $a$ and $b$ be positive integers. Then there exists a graph $G$ such that $a=\gamma^{LD}(G)$ and $b=\gamma^{DLD}(G)$ if and only if we have $$0\leq b-a\leq 2^a-1-\binom{a}{\left\lceil\frac{a}{2}\right\rceil}.$$
\end{theorem}
\begin{proof}
{Let us have a locating-dominating code $C_{LD}$ of cardinality $a$ in $G=(V,E)$. Then all the $|V|-a$ non-codewords  $u\in V$ have different and non-empty sets $I(C_{LD};u)$. Hence, we have $|V|\leq 2^a-1+a$. Denote  $R=\left\{v\in V\mid |N[v]\cap C_{LD}|= \left\lceil \frac{a}{2}\right\rceil, v\not\in C_{LD} \right\}$. Clearly, $0\le|R|\leq \binom{a}{\left\lceil\frac{a}{2}\right\rceil}$ and there are at most $2^a-1+a- \binom{a}{\left\lceil\frac{a}{2} \right\rceil}$ vertices in $V\setminus R$ (as the $I$-sets of the non-codewords have to be non-empty and unique). Denote then $D = V \setminus R$. Now for any distinct pair of vertices $v, w\in R$ we have $I(D;v)\not\subseteq I(D;w)$ because $C_{LD}\subseteq D$ and $|I(C_{LD};v)| = |I(C_{LD};w)| = \lceil a/2 \rceil$. Therefore, we have $b = \gamma^{DLD}(G) \leq |D| \leq 2^a-1+a- \binom{a}{\left\lceil\frac{a}{2} \right\rceil}$. Thus, we obtain that $b-a \leq 2^a-1- \binom{a}{\left\lceil\frac{a}{2} \right\rceil}$.
}


Let us first consider the situation $a=b$ and a graph $G=(V,E)$. We notice that this is possible by assuming that $G$ is the star with $a$ pendant vertices. Let us then assume that $a<b$ and let us have $k\in \Z_+$ such that $2^{k-1}-1-\binom{k-1}{\left\lceil\frac{k-1}{2}\right\rceil}< b-a\leq2^{k}-1-\binom{k}{\left\lceil\frac{k}{2}\right\rceil}$. Since $2^{k}-1-\binom{k}{\left\lceil\frac{k}{2}\right\rceil}$ is an increasing function on $k$ when $k>0$ and it gains value $1$ when $k=2$, each value of difference $b-a$ is linked to a unique value of $k$. Since the function gives $0$ when $k=1$, we can assume that $k\geq2$. Let us have $K=\{v_i\mid1\leq i\leq k\}$, $K'=\{v_i\mid k+1\leq i\leq a\}$, $P=\left\{u_i\mid1\leq i\leq b-a+\binom{k}{\left\lceil\frac{k}{2}\right\rceil}\right\}$ and $V=K\cup K'\cup P$. We have $|V|=b+\binom{k}{\left\lceil\frac{k}{2}\right\rceil}$. Let us connect
\begin{enumerate}
\item $v_i$ to $v_j$ for each $1\leq i<j\leq a$ forming the complete graph $K_a$,
\item $u_1$ to $v_i$ for each $1\leq i\leq a$,
\item each $u_i$, $2\leq i\leq \binom{k}{\left\lceil\frac{k}{2}\right\rceil}+1$, to $\left\lfloor\frac{k}{2}\right\rfloor$ vertices in $K$ in such a manner that all of these vertices have different open neighbourhoods,
\item each $u_i$, $\binom{k}{\left\lceil\frac{k}{2}\right\rceil}+2\leq i\leq \binom{k}{\left\lceil\frac{k}{2}\right\rceil}+k+1$, to vertex $v_j$ with $j=i-\binom{k}{\left\lceil\frac{k}{2}\right\rceil}-1$ when $k\geq 4$ (and no vertices are connected if $2 \leq k \leq 3$) and
\item each other vertex $u_i\in P$ to some non-empty subset of vertices in $K$ in such a manner that no two vertices of $P$ have the same neighbourhood in $K$.
\end{enumerate}

Denote the graph constructed above by $G$. Step $3$ is possible since $\binom{k}{\left\lceil\frac{k}{2} \right\rceil}=\binom{k}{\left\lfloor\frac{k}{2}\right\rfloor}$ and $|P|\geq\binom{k}{\left\lceil\frac{k}{2}\right\rceil}+1$. Furthermore, because $k\geq 2$, $u_1$ has different neighbourhood in $K$ (compared to the vertices of Step~3). Step $4$ is possible since $$|P|\geq 2^{k-1}-1- \binom{k-1}{\left\lceil\frac{k-1}{2}\right\rceil}+\binom{k}{\left\lceil\frac{k}{2}\right\rceil}+1 = \sum_{i=0}^{k-1}\binom{k-1}{i}- \binom{k-1}{\left\lceil\frac{k-1}{2}\right\rceil}+\binom{k}{\left\lceil\frac{k}{2}\right\rceil}\geq \binom{k}{\left\lceil\frac{k}{2}\right\rceil}+k+1$$ when $k\geq4$. However, when $k=2$ or $k=3$, we have $\left\lfloor\frac{k}{2}\right\rfloor=1$ and thus, by step $3$, also in these cases for each vertex $v\in K$ there exists a vertex of $P$ such that it has only $v$ in its neighbourhood. Step $5$ is possible since we have $|P|\leq2^{|K|}-1 = 2^k-1$.

Let us first show that the location-domination number of $G$ is equal to $a$. We can now choose $K\cup K'$ as a locating-dominating code of size $a$ since each vertex in $P$ has its open neighbourhood with a unique and non-empty intersection with $K$. Furthermore, each vertex in $K'\cup \{u_1\}$ has the same closed neighbourhood and hence, we can have at most one non-codeword among them. Since each vertex $v_j\in K$ neighbours a vertex $u_i$ such that $N(u_i)=\{v_j\}$ for some $j$ and $i$, we have $v_j$ or $u_i$ in the code and hence, we have at least $|K|+|K'|=a$ codewords. Thus, $\gamma^{LD}(G) = a$.

Let us then show that $\gamma^{DLD}(G) = b$. We can choose as our solid-locating-dominating code $C$ the vertex set containing all vertices except for $u_i$ for $2\leq i\leq \binom{k}{\left\lceil\frac{k}{2}\right\rceil}+1$. We have $|C|=b$ (as $|V| = b + \binom{k}{\lceil k/2 \rceil}$) and if $v,v'\not\in C$, then $|N(v)\cap K|=|N(v')\cap K|=\left\lfloor\frac{k}{2}\right\rfloor$ and $N(v)\cap K\neq N(v')\cap K$ and hence, $C$ is a solid-locating-dominating code. In order to show that $\gamma^{DLD}(G) \geq b$, let $X$ be a solid-locating-dominating code in $G$. Observe that if a vertex $v_h \in K'\cup K$ does not belong to $X$, then each vertex in $P$ is in the code since for $u_i\in P$ we have $N(u_i)\subseteq K'\cup K\subseteq N[v_h]$. Moreover, since we have $N[v']\subseteq N[v]$ when $v'\in K'$ and $v\in K'\cup K$, we can have at most one non-codeword in $K'$ and only if there are no non-codewords in $K$. Therefore, we may assume that $K\cup K' \subseteq X$ since otherwise we would have at least $|P|+|K'|=b+\binom{k}{\left\lceil\frac{k}{2}\right\rceil}-k\geq b$ codewords because $\binom{k}{\left\lceil\frac{k}{2}\right\rceil}\geq k$. Furthermore, due to the Sperner's theorem and the independence of the set $P$, we can choose at most $\binom{k}{\left\lceil\frac{k}{2}\right\rceil}$ vertices in $P$ in such a manner that none of their neighbourhoods is contained within another. This gives $\gamma^{DLD}(G)\geq b$.\end{proof}




\end{document}